\documentclass[reqno]{amsart}

\usepackage{mathrsfs}
\usepackage{amsmath}
\usepackage{amssymb}
\usepackage{cite}
\usepackage{latexsym}
\usepackage{graphicx}

\usepackage{color,comment}

\theoremstyle{plain}
\newtheorem{thm}{Theorem}[section]

\newtheorem{cor}[thm]{Corollary}
\newtheorem{dfn}[thm]{Definition}
\newtheorem{prop}[thm]{Proposition}
\newtheorem{rmk}[thm]{Remark}

\def\D{\mathrm{D}}
\def\G{\mathcal{G}}

\def\R{\mathscr{R}}

\def\c{\mathrm{c}}
\def\d{\mathrm{d}}

\def\n{\mathrm{n}}
\def\r{\mathrm{R}}

\def\Cset{\mathbb{C}}
\def\Nset{\mathbb{N}}
\def\Qset{\mathbb{Q}}
\def\Rset{\mathbb{R}}

\def\Zset{\mathbb{Z}}

\def\Span{\mathrm{span}}

\def\epsilon{\varepsilon}

\makeatletter
 \@addtoreset{equation}{section}
\makeatother
\def\theequation{\arabic{section}.\arabic{equation}}

\begin{document}


\title[Nonintegrability of truncated normal forms]%
{Nonintegrability of truncated Poincar\'e-Dulac normal forms of resonance degree two}

\author{Kazuyuki Yagasaki}

\address{Department of Applied Mathematics and Physics, Graduate School of Informatics,
Kyoto University, Yoshida-Honmachi, Sakyo-ku, Kyoto 606-8501, JAPAN}
\email{yagasaki@amp.i.kyoto-u.ac.jp}

\date{\today}

\begin{abstract}
We give sufficient conditions
 for three- or four-dimensional truncated Poincar\'e-Dulac normal forms
 of resonance degree two
 to be meromorphically nonintegrable
 when the Jacobian matrices have a zero and pair of purely imaginary eigenvalues
 or two incommensurate pairs of purely imaginary eigenvalues at the equilibria.
For this purpose, we reduce their integrability
 to that of simple planar systems,
 and use an approach for proving the meromorphic nonintegrability of planar systems,
 which is similar to but more sophisticated than the previously developed one. 
Our result also implies that general three- or four-dimensional systems
 are analytically nonintegrable
 if they are formally transformed into one of the truncated normal forms
 satisfying the sufficient conditions.
\end{abstract}
\maketitle


\section{Introduction}

Consider systems of the general form
\begin{equation}
\dot{x}=f(x),\quad
x\in\Rset^n,
\label{eqn:sys}
\end{equation}
for $n\in\Nset$,
 where $f(x)$ is analytic and $x=0$ is an equilibrium, i.e., $f(0)=0$.
Let $\lambda_j$,  $j=1,\ldots,n$, be eigenvalues of the Jacobian matrix $\D f(0)$ at $x=0$.
The following two concepts are important in this situation.

\begin{dfn}[Poincar\'e-Dulac normal form]
Change the coordinates in \eqref{eqn:sys} such that
 $\D f(0)$ is in Jordan normal form.
The system \eqref{eqn:sys} is called a \emph{Poincar\'e-Dulac (PD) normal form} if $[Sx,f]=0$,
 where $S$ is the semisimple part of $\D f(0)$, i.e., $S=\mathrm{diag}\lambda_j$,
 where $\lambda_j$, $j=1,\ldots,n$, are the eigenvalues of $\D f(0)$.
\end{dfn}

Let
\[
\Zset_j^n=\{p=(p_1,\ldots,p_n)\in\Zset^n
 \mid p_j\ge -1,\,p_l\ge 0,\,l\neq j,\,p\neq 0\}
\]
for $j=1,\dots,n$.

\begin{dfn}[Resonance set and degree]
Let
\[
\R_j=\left\{p\in\Zset_j^n\,\left|\,\sum_{l=1}^n\lambda_jp_j=0\right.\right\},\quad
j=1,\ldots,n,
\]
and let
\[
\R=\bigcup_{j=1}^n\R_j.
\]
We refer to $\R$ as the \emph{resonance set} of \eqref{eqn:sys} and to
\[
\gamma_\r=\dim_\Qset\Span_\Qset\R
\]
as the \emph{resonance degree}  of \eqref{eqn:sys}. 
\end{dfn}

See \cite{B89,Z02,Y18b} for more details.
Yamanaka \cite{Y18b} proved the following. 

\begin{thm}[Yamanaka]
\label{thm:yama}
If the resonance degree $\gamma_\r$ is less than two,
 then the PD normal form is analytically integrable.
Moreover, there exists an $n$-dimensional, analytically nonintegrable PD normal form
 with $n=\gamma_\r+1$ for $\gamma_\r\ge 2$.
\end{thm}

Here we adopt the following concept of integrability for such general systems as \eqref{eqn:sys}
 in the Bogoyavlenskij sense \cite{B98}.

\begin{dfn}[Bogoyavlenskij]
\label{dfn:1a}
Let $q$ be an integer such that $0\le q\le n$.
The $n$-dimensional system \eqref{eqn:sys}
 is called \emph{$(q,n-q)$-integrable} or simply \emph{integrable} 
 if there exist $q$ vector fields $f_1(x)(:= f(x)),f_2(x),\dots,f_q(x)$
 and $n-q$ scalar-valued functions $F_1(x),\dots,F_{n-q}(x)$ such that
 the following two conditions hold:
\begin{enumerate}
\setlength{\leftskip}{-1.8em}
\item[\rm(i)]
$f_1(x),\dots,f_q(x)$ are linearly independent almost everywhere
 and commute with each other,
 i.e., $[f_j,f_k](x):=\D f_k(x)f_j(x)-\D f_j(x)f_k(x)\equiv 0$ for $j,k=1,\ldots,q$,
 where $[\cdot,\cdot]$ denotes the Lie bracket$\hspace{0.05em};$
\item[\rm(ii)]
The derivatives $\D F_1(x),\dots, \D F_{n-q}(x)$ are linearly independent almost everywhere
 and $F_1(x),\dots,F_{n-q}(x)$ are first integrals of $f_1, \dots,f_q$,
 i.e., $\D F_k(x)\cdot f_j(x)\equiv 0$ for $j=1,\ldots,q$ and $k=1,\ldots,n-q$,
 where ``$\cdot$'' represents the inner product.
\end{enumerate}
We say that the system is \emph{meromorphically} $($resp. \emph{analytically}$)$
 \emph{integrable}
 if the first integrals and commutative vector fields are meromorphic $($resp. analytic$)$. 
\end{dfn}

Definition~\ref{dfn:1a} is considered as a generalization of 
 Liouville-integrability for Hamiltonian systems \cite{A89,M99}
 since an $m$-degree-of-freedom Liouville-integrable Hamiltonian system with $m\ge 1$
 has not only $m$ functionally independent first integrals
 but also $m$ linearly independent commutative (Hamiltonian) vector fields
 generated by the first integrals.
Results similar to Theorem~\ref{thm:yama}
  for Hamiltonian systems are found in \cite{C12,D84,Y19,Z05}.

Moreover, Zung \cite{Z02} proved the following.

\begin{thm}[Zung]
\label{thm:zung}
If the system \eqref{eqn:sys} is analytically integrable in the Bogoyavlenskij sense,
 then there exists an analytic change of coordinates
 under which it is transformed to a PD normal form.
\end{thm}

A similar result for Hamiltonian systems was also obtained in \cite{Z05}.

In this paper, we study \eqref{eqn:sys} with $n=3$ or $4$
 for which $\D f(0)$ of $f(x)$ at $x=0$ has
 (I) a zero and a pair of purely imaginary eigenvalues, $\lambda=0,\pm i\omega$ ($\omega>0$),
 for $n=3$
 or (II) two pairs of  purely imaginary eigenvalues, $\lambda=\pm i\omega_j$ ($\omega_j>0$), $j=1,2$,
 with $\omega_1/\omega_2\not\in\Qset$ for $n=4$.
Then by polynomial changes of coordinates,
 the system \eqref{eqn:sys} is transformed to
\begin{equation}
\begin{split}
\dot{x}_1=&-\omega x_2+\alpha_1x_1x_3-\alpha_2 x_2x_3\\
& +(\beta_1(x_1^2+x_2^2)+\beta_2x_3^2)x_1-(\beta_4(x_1^2+x_2^2)+\beta_5x_3^2)x_2,\\
\dot{x}_2=&\omega x_1+\alpha_2 x_1x_3+\alpha_1 x_2x_3\\
& +(\beta_4(x_1^2+x_2^2)+\beta_5x_3^2)x_1+(\beta_1(x_1^2+x_2^2)+\beta_2x_3^2)x_2,\\
\dot{x}_3=& \alpha_3(x_1^2+x_2^2)+\alpha_4x_3^2+(\beta_5(x_1^2+x_2^2)+\beta_6 x_3)x_3^2
\end{split}
\label{eqn:PD3}
\end{equation}
with $x=(x_1,x_2,x_3)$ up to $O(|x|^3)$ for case (I), and to
\begin{equation}
\begin{split}
\dot{x}_1
=& -\omega_1 x_2+(\alpha_1(x_1^2+x_2^2)+\alpha_2(x_3^2+x_4^2))x_1
 -\tilde{\alpha}_1(x_1^2+x_2^2)+\tilde{\alpha}_2(x_3^2+x_4^2))x_2\\
& +(\beta_1(x_1^2+x_2^2)^2+\beta_2(x_3^2+x_4^2)^2
 +\beta_3(x_1^2+x_2^2)(x_3^2+x_4^2))x_1\\
& -(\tilde{\beta}_1(x_1^2+x_2^2)^2+\tilde{\beta}_2(x_3^2+x_4^2)^2
 +\tilde{\beta}_3(x_1^2+x_2^2)(x_3^2+x_4^2))x_2,\\
\dot{x}_2
=& \omega_1 x_1+(\tilde{\alpha}_1(x_1^2+x_2^2)+\tilde{\alpha}_2(x_3^2+x_4^2))x_1
 +(\alpha_1(x_1^2+x_2^2)+\alpha_2(x_3^2+x_4^2))x_2\\
& +(\tilde{\beta}_1(x_1^2+x_2^2)^2+\tilde{\beta}_2(x_3^2+x_4^2)^2
 +\tilde{\beta}_3(x_1^2+x_2^2)(x_3^2+x_4^2))x_1\\
& +(\beta_1(x_1^2+x_2^2)^2+\beta_2(x_3^2+x_4^2)^2
 +\beta_3(x_1^2+x_2^2)(x_3^2+x_4^2))x_2,\\
\dot{x}_3
=& -\omega_2 x_4+(\alpha_3(x_1^2+x_2^2)+\alpha_4(x_3^2+x_4^2)^2)x_3
 -(\tilde{\alpha}_3(x_1^2+x_2^2)+\tilde{\alpha}_4(x_3^2+x_4^2))x_4\\
& +(\beta_4(x_1^2+x_2^2)^2+\beta_5(x_3^2+x_4^2)^2
 +\beta_6(x_1^2+x_2^2)(x_3^2+x_4^2))x_3\\
& -(\tilde{\beta}_4(x_1^2+x_2^2)^2+\tilde{\beta}_5(x_3^2+x_4^2)^2
 +\tilde{\beta}_6(x_1^2+x_2^2)(x_3^2+x_4^2))x_4,\\
\dot{x}_4
=& \omega_2 x_3+(\tilde{\alpha}_3(x_1^2+x_2^2)+\tilde{\alpha}_4(x_3^2+x_4^2))x_3
 +(\alpha_3(x_1^2+x_2^2)+\alpha_4(x_3^2+x_4^2))x_4\\
& +(\tilde{\beta}_4(x_1^2+x_2^2)^2+\tilde{\beta}_5(x_3^2+x_4^2)^2
 +\tilde{\beta}_6(x_1^2+x_2^2)(x_3^2+x_4^2))x_3\\
& +(\beta_4(x_1^2+x_2^2)^2+\beta_5(x_3^2+x_4^2)^2
 +\beta_6(x_1^2+x_2^2)(x_3^2+x_4^2))x_4
\end{split}
\label{eqn:PD4}
\end{equation}
with $x=(x_1,x_2,x_3,x_4)$ up to $O(|x|^5)$ for case (II),
 where $\alpha_j,\beta_j,\tilde{\alpha}_j,\tilde{\beta}_j\in\Rset$, $j=1,\ldots,4$ or $1,\ldots,6$.
See, e.g., Section~3.1 of \cite{HI11} for the derivations of \eqref{eqn:PD3} and \eqref{eqn:PD4}.
We easily see that the systems \eqref{eqn:PD3} and \eqref{eqn:PD4}
 can be written as the $O(|x|^3)$- and $O(|x|^5)$-truncations
 of the PD normal forms in cases (I) and (II), respectively.
Actually, Eqs.~\eqref{eqn:PD3} and \eqref{eqn:PD4} become
\begin{equation}
\begin{split}
&
\dot{z}_1=i\omega z_1+(\alpha_1x_3+\beta_1z_1z_2+\beta_2x_3^2
 +i(\alpha_2x_3+\beta_4z_1z_2+\beta_5x_3^2))z_1,\\
&
\dot{z}_2=-i\omega z_2+(\alpha_1x_3+\beta_1z_1z_2+\beta_2x_3^2
 -i(\alpha_2x_3+\beta_4z_1z_2+\beta_5x_3^2))z_2,\\
&
\dot{x}_3=\alpha_3z_1z_2+\alpha_4x_3^2+(\beta_5z_1z_2+\beta_6x_3^2)x_3
\end{split}
\label{eqn:cpd1}
\end{equation}
and
\begin{equation}
\begin{split}
\dot{z}_1=&i\omega_1 z_1
 +((\alpha_1+i\tilde{\alpha}_1)z_1z_2+(\alpha_2+i\tilde{\alpha}_2)z_3z_4\\
&
+(\beta_1+i\tilde{\beta}_1)z_1^2z_2^2+(\beta_2+i\tilde{\beta}_2)z_3^2z_4^2
 +(\beta_3+i\tilde{\beta}_3)z_1z_2z_3z_4 )z_1,\\
\dot{z}_2=&-i\omega_1 z_2
 +((\alpha_1-i\tilde{\alpha}_1)z_1z_2+(\alpha_2-i\tilde{\alpha}_2)z_3z_4\\
&
+(\beta_1-i\tilde{\beta}_1)z_1^2z_2^2+(\beta_2-i\tilde{\beta}_2)z_3^2z_4^2
 +(\beta_3-i\tilde{\beta}_3)z_1z_2z_3z_4)z_2,\\
\dot{z}_3=&i\omega_2 z_3
 +((\alpha_3+i\tilde{\alpha}_3)z_1z_2+(\alpha_4+i\tilde{\alpha}_4)z_3z_4\\
&
+(\beta_4+i\tilde{\beta}_4)z_1^2z_2^2+(\beta_5+i\tilde{\beta}_5)z_3^2z_4^2
 +(\beta_6+i\tilde{\beta}_6)z_1z_2z_3z_4)z_3,\\
\dot{z}_4=&-i\omega_2 z_4
 +((\alpha_3-i\tilde{\alpha}_3)z_1z_2+(\alpha_4-i\tilde{\alpha}_4)z_3z_4\\
&
+(\beta_4-i\tilde{\beta}_4)z_1^2z_2^2+(\beta_5-i\tilde{\beta}_5)z_3^2z_4^2
 +(\beta_6-i\tilde{\beta}_6)z_1z_2z_3z_4)z_4,
\end{split}
\label{eqn:cpd2}
\end{equation}
respectively, where
\begin{align*}
z_1=x_1+ix_2,\quad
z_2=x_1-ix_2,\quad
z_3=x_3+ix_4,\quad
z_4=x_3-ix_4.
\end{align*}
We easily see that the systems \eqref{eqn:cpd1} and \eqref{eqn:cpd2} are PD normal forms.
Moreover, the resonance sets are given by
\[
\R=\Span_\Nset\{(1,0,0),(0,1,1)\}\quad\text{and}\quad
\R=\Span_\Nset\{(1,1,0,0),(0,0,1,1)\},
\]
respectively, and the resonance degrees are $\gamma_\r=2$.
Hence, the systems \eqref{eqn:PD3} and \eqref{eqn:PD4} may be analytically nonintegrable
 near the origin $x=0$ even in reference to Theorem~\ref{thm:yama}.
Actually, the following were recently proved in \cite{Y22}.

\begin{thm}
\label{thm:Y1}
Let $n=3$ and suppose that the system~\eqref{eqn:sys} is transformed to
\begin{equation}
\begin{split}
\dot{x}_1=&-\omega x_2+\alpha_1x_1x_3-\alpha_2 x_2x_3,\\
\dot{x}_2=&\omega x_1+\alpha_2 x_1x_3+\alpha_1 x_2x_3,\\
\dot{x}_3=& \alpha_3(x_1^2+x_2^2)+\alpha_4x_3^2
\end{split}
\label{eqn:PD3t}
\end{equation}
up to $O(|x|^2)$.
If one of the following conditions holds,
 then the system~\eqref{eqn:sys}
 is not real-analytically integrable in the Bogoyavlenskij sense near the origin$\,:$
\begin{enumerate}
\setlength{\leftskip}{-1.8em}
\item[\rm(i)]
$\alpha_1\alpha_4>0\,;$
\item[\rm(ii)]
$\alpha_1\alpha_4<0$ and $\alpha_4/\alpha_1\not\in\Qset$.
\end{enumerate}
\end{thm}

\begin{thm}
\label{thm:Y2}
Let $n=4$ and suppose that the system~\eqref{eqn:sys} is transformed to
\begin{equation}
\begin{split}
\dot{x}_1
=& -\omega_1 x_2+(\alpha_1(x_1^2+x_2^2)+\alpha_2(x_3^2+x_4^2))x_1
 -(\tilde{\alpha}_1(x_1^2+x_2^2)+\tilde{\alpha}_2(x_3^2+x_4^2))x_2,\\
\dot{x}_2
=& \omega_1 x_1+(\tilde{\alpha}_1(x_1^2+x_2^2)+\tilde{\alpha}_2(x_3^2+x_4^2))x_1
 +(\alpha_1(x_1^2+x_2^2)+\alpha_2(x_3^2+x_4^2))x_2,\\
\dot{x}_3
=& -\omega_2 x_4+(\alpha_3(x_1^2+x_2^2)+\alpha_4(x_3^2+x_4^2))x_3
 -(\tilde{\alpha}_3(x_1^2+x_2^2)+\tilde{\alpha}_4(x_3^2+x_4^2))x_4,\\
\dot{x}_4
=& \omega_2 x_3+(\tilde{\alpha}_3(x_1^2+x_2^2)+\tilde{\alpha}_4(x_3^2+x_4^2))x_3
 +(\alpha_3(x_1^2+x_2^2)+\alpha_4(x_3^2+x_4^2))x_4
\end{split}
\label{eqn:PD4t}
\end{equation}
up to $O(|x|^3)$.
If $\alpha_1\neq\alpha_3$, $\alpha_2\neq\alpha_4$, and one of the following conditions holds,
 then the system~\eqref{eqn:sys} is not real-analytically integrable
 in the Bogoyavlenskij sense near the origin$\,:$
\begin{enumerate}
\setlength{\leftskip}{-1.6em}
\item[\rm(i)]
$\alpha_1\alpha_3$ or $\alpha_2\alpha_4>0\,;$
\item[\rm(ii)]
$\alpha_1\alpha_3,\alpha_2\alpha_4<0$
 and $\alpha_1/\alpha_3,\alpha_2/\alpha_4\not\in\Qset$.
\end{enumerate}
\end{thm}

From Theorem~\ref{thm:zung} we see that
 when the hypotheses of Theorems~\ref{thm:Y1} and \ref{thm:Y2} hold respectively,
 the truncated systems \eqref{eqn:PD3} and \eqref{eqn:PD4} are analytically nonintegrable,
 and the system \eqref{eqn:sys} is analytically nonintegrable
 if it is formally transformed into one of them.
So we are interested in the following questions: 
\begin{itemize}
\setlength{\leftskip}{-2.8em}
\item
Can the truncated systems \eqref{eqn:PD3} and \eqref{eqn:PD4} be nonintegrable
 even when the hypotheses of Theorem~\ref{thm:Y1} or \ref{thm:Y2} are not satisfied?
\item
If so, under what conditions are they nonintegrable?
\end{itemize}
Here we give sufficient conditions for the nonintegrability
 of the truncated systems \eqref{eqn:PD3} and \eqref{eqn:PD4}
 and answers to the above questions.
Our results are also important in the following context.

The unfoldings of the $O(|x|^2)$- and $O(|x|^3)$-truncated systems
 for \eqref{eqn:PD3} and \eqref{eqn:PD4},
\begin{equation}
\begin{split}
&\dot{x}_1=\nu x_1-\omega x_2+\alpha_1 x_1x_3-\alpha_2 x_2x_3,\\
&\dot{x}_2=\omega x_1+\nu x_2+\alpha_2 x_1x_3+\alpha_1 x_2x_3,\\
&\dot{x}_3=\mu+\alpha_3(x_1^2+x_2^2)+\alpha_4x_3^2
\end{split}
\label{eqn:fH}
\end{equation}
and
\begin{equation}
\begin{split}
\dot{x}_1
=& -\omega_1 x_2+(\nu+\alpha_1(x_1^2+x_2^2)+\alpha_2(x_3^2+x_4^2))x_1,\\
\dot{x}_2
=& \omega_1 x_1+(\nu+\alpha_1(x_1^2+x_2^2)+\alpha_2(x_3^2+x_4^2))x_2,\\
\dot{x}_3
=& -\omega_2 x_4+(\mu+\alpha_3(x_1^2+x_2^2)+\alpha_4(x_3^2+x_4^2))x_3,\\
\dot{x}_4
=& \omega_2 x_3+(\mu+\alpha_3(x_1^2+x_2^2)+\alpha_4(x_3^2+x_4^2))x_4,
\end{split}
\label{eqn:dH}
\end{equation}
represent normal forms of fold-Hopf and double-Hopf bifurcations, respectively,
 where $\mu,\nu\in\Rset$ are the control parameters:
At $(\mu,\nu)=(0,0)$, fold (saddle-node) and Hopf bifurcation curves meet for the former
 and two Hopf bifurcation curves for the latter.
Such codimension-two bifurcations are fundamental and interesting phenomena
 in dynamical systems
 and have been studied extensively \cite{GH83,HI11,K04}
 since the seminal papers of Arnold \cite{A72} and Takens \cite{T74}.
In \cite{AY20,Y18a},
 the integrability of the normal forms \eqref{eqn:fH} and \eqref{eqn:dH}
 in the Bogoyavlenskij sense were discussed:
 They were shown to be meromorphically nonintegrable
 for almost all parameter values of $\alpha_j$, $j=1,2$,
 near the $x_3$-axis and the $(x_1,x_2)$- or $(x_3,x_4)$-plane, respectively,
 when $(\mu,\nu)\neq(0,0)$,
 while it was not determined there
 whether they are meromorphically nonintegrable or not when $(\mu,\nu)=(0,0)$.
 
We now state our main results.
Let $\Zset_{\ge l}=\{j\in\Zset\mid j\ge l\}$.

\begin{thm}
\label{thm:main1}
If one of the following conditions holds, then the system~\eqref{eqn:PD3}
 is not meromorphiaclly integrable in the Bogoyavlenskij sense near the $x_3$-axis$:$
\begin{enumerate}
\setlength{\leftskip}{-1em}
\item[\rm(i)]
$\alpha_1\alpha_3\alpha_4\beta_6,
 \beta_2/\beta_6-\alpha_1/\alpha_3,\beta_5/\beta_6-\alpha_4/\alpha_3\neq 0$,
 $\alpha_1/\alpha_3$ or $\beta_2/\beta_6\not\in\Qset$, and
\begin{enumerate}
\setlength{\leftskip}{0.1em}
\setlength{\labelwidth}{5.6em}
\item[\rm(ia)]
$\alpha_1/\alpha_3=1$
 or $\beta_2/\beta_6-\alpha_1/\alpha_3=\tfrac{1}{2};$
\item[\it or \rm(ib)]
$2\alpha_1/\alpha_3\not\in\Zset_{\ge 2}$,
 $2(\beta_2/\beta_6-\alpha_1/\alpha_3)\not\in\Zset_{\ge 1}$, and
$\beta_1\beta_6-\beta_2\beta_5\neq 0$ or
\begin{equation}
2(\alpha_1-\alpha_3)(\alpha_1\beta_5-\alpha_3\beta_1)
 +\alpha_4(3\alpha_1\beta_6-2\alpha_3\beta_2)\neq 0;
\label{eqn:thm1a}
\end{equation}
\end{enumerate}
\item[\rm(ii)]
$\alpha_1\alpha_3\alpha_4\beta_6,
 \alpha_1\beta_5-\alpha_3\beta_1-\alpha_4\beta_2\neq 0$
 and $\alpha_1/\alpha_3=\beta_2/\beta_6\not\in\Qset;$
\item[\rm(iii)]
$\alpha_1\alpha_3\alpha_4\beta_6\neq 0$,
 $\beta_5/\beta_6=\alpha_4/\alpha_3$,
 $\alpha_1/\alpha_3$ or $\beta_2/\beta_6\not\in\Qset$, and
\begin{enumerate}
\setlength{\leftskip}{0.8em}
\setlength{\labelwidth}{5.6em}
\item[\rm(iiia)]
$\alpha_1\beta_5+\alpha_3\beta_1-\alpha_4\beta_2\neq 0$, and
\begin{enumerate}
\setlength{\leftskip}{2.5em}
\setlength{\labelwidth}{5.6em}
\item[\rm(iiia1)]
$\alpha_1/\alpha_3=1$ or $\beta_2/\beta_6-\alpha_1/\alpha_3=\tfrac{1}{2}$
\item[\it or \rm(iiia2)]
$2\alpha_1/\alpha_3\not\in\Zset_{\ge 2}$, $2(\beta_2/\beta_6-\alpha_1/\alpha_3)\not\in\Zset_{\ge 1}$ and
\begin{equation}
\alpha_1\beta_6(2\beta_1-3\beta_5)-2\alpha_3(\beta_1\beta_6-\beta_2\beta_5)\neq 0;
\label{eqn:thm1b}
\end{equation}
\end{enumerate}
\item[\it or \rm(iiib)]
$\alpha_1\beta_5+\alpha_3\beta_1-\alpha_4\beta_2=0$, and
\begin{enumerate}
\setlength{\leftskip}{2.5em}
\setlength{\labelwidth}{5.6em}
\item[\rm(iiib1)]
$\alpha_1/\alpha_3=1;$
\item[\it or \rm(iiib2)]
$2\alpha_1/\alpha_3\not\in\Zset_{\ge 2}$ and $\beta_2\neq\beta_6;$
\end{enumerate}
\end{enumerate}
\item[\rm(iv)]
$\alpha_1\alpha_3\beta_6,\beta_2/\beta_6-\alpha_1/\alpha_3\neq 0$, $\alpha_4=0$, 
 $\alpha_1/\alpha_3$ or $\beta_2/\beta_6\not\in\Qset$, and
\begin{enumerate}
\setlength{\leftskip}{0.7em}
\setlength{\labelwidth}{5.6em}
\item[\rm(iva)]
$\beta_5,\alpha_3\beta_1-\alpha_1\beta_5\neq 0$, and
\begin{enumerate}
\setlength{\leftskip}{2.4em}
\setlength{\labelwidth}{5.6em}
\item[\rm(iva1)]
$\alpha_1/\alpha_3$ or $\beta_2/\beta_6-\alpha_1/\alpha_3=\tfrac{1}{2};$
\item[\it or \rm(iva2)]
 $2\alpha_1/\alpha_3,2(\beta_2/\beta_6-\alpha_1/\alpha_3)\not\in\Zset_{\ge 1}$ and
\begin{equation}
2\alpha_1\beta_6(\beta_1-\beta_5)
 +\alpha_3(\beta_1\beta_6+\beta_2\beta_5-2\beta_1\beta_2)\neq 0;
\label{eqn:thm1c}
\end{equation}
\end{enumerate}
\item[\rm(ivb)]
$\beta_5\neq 0$, $\beta_1/\beta_5=\alpha_1/\alpha_3$, 
 $\beta_2/\beta_6\neq\tfrac{1}{2}$
 and $2(\beta_2/\beta_6-\alpha_1/\alpha_3)\not\in\Zset_{\ge 2};$
\item[\it or \rm(ivc)]
$\beta_1\neq 0$, $\beta_5=0$, $\beta_2/\beta_6\neq\tfrac{1}{2}$
 and $2\alpha_1/\alpha_3\not\in\Zset_{\ge 2};$
\end{enumerate}
\item[\rm(v)]
$\alpha_1\alpha_3\beta_6,\alpha_3\beta_1-\alpha_1\beta_5\neq 0$, $\alpha_4=0$
 and $\alpha_1/\alpha_3=\beta_2/\beta_6\not\in\Qset;$
\item[(vi)]
$\alpha_3\beta_6\neq 0$, $\alpha_1=0$, $\beta_2/\beta_6\not\in\Qset$ and
\begin{enumerate}
\setlength{\leftskip}{0.7em}
\setlength{\labelwidth}{5.6em}
\item[\rm(via)]
$\beta_5/\beta_6-\alpha_4/\alpha_3,\beta_1/\beta_2-\alpha_4/\alpha_3\neq 0$ and
\begin{equation}
\alpha_3(2\beta_1\beta_2-\beta_2\beta_5-\beta_1\beta_6)
 -2\alpha_4\beta_2(\beta_2-\beta_6)\neq 0;
\label{eqn:thm1d}
\end{equation}
\item[\rm(vib)]
$\beta_1/\beta_2=\alpha_4/\alpha_3$ and $\beta_5/\beta_6\neq\alpha_4/\alpha_3;$ 
\item[\it or \rm(vic)]
$\beta_5/\beta_6\neq\alpha_4/\alpha_3$ and $\beta_1/\beta_2=\alpha_4/\alpha_3;$
\end{enumerate}
\item[\rm(vii)]
$\alpha_1\alpha_4\beta_6\neq 0$ and $\alpha_3=0$,
 and $\beta_1\beta_6-\beta_2\beta_5$
  or $2\alpha_1\beta_5+3\alpha_4\beta_6\neq 0;$
\item[\rm(viii)]
$\alpha_1\beta_6\neq 0$, $\alpha_3,\alpha_4=0$,
 $\beta_1\neq\beta_5$, and 
\begin{enumerate}
\setlength{\leftskip}{1.2em}
\setlength{\labelwidth}{5.6em}
\item[\rm(viiia)]
$\beta_5\neq 0;$
\item[\it or \rm(viiib)]
$\beta_5=0$ and $\beta_2/\beta_6\neq\tfrac{1}{2}$.
\end{enumerate}
\end{enumerate}
\end{thm}

We give a proof of Theorem~\ref{thm:main1} in Section~4.

\begin{rmk}
Assume that the hypotheses of Theorem~$\ref{thm:Y1}$ do  not hold, i.e.,
\begin{enumerate}
\setlength{\leftskip}{-0.6em}
\item[\rm(i)]
$\alpha_1\alpha_4=0;$
\item[or \rm(ii)]
$\alpha_1\alpha_4<0$ and $\alpha_4/\alpha_1\in\Qset$.
\end{enumerate}
In case {\rm(i)},
 conditions~{\rm(iv)}-{\rm(vi)} and {\rm(viii)} of Theorem~$\ref{thm:main1}$
 may hold.
In case {\rm(ii)},
 conditions~{\rm(i)}-{\rm(iii)} and {\rm(vii)} of Theorem~$\ref{thm:main1}$
 may hold.
Thus, even if the hypotheses of Theorem~$\ref{thm:Y1}$ do  not hold,
 the system~\eqref{eqn:PD3} may be meromorphically nonintegrable.
So Theorem~$\ref{thm:main1}$ gives answers to the questions stated above for \eqref{eqn:PD3}.
\end{rmk}

\begin{thm}
\label{thm:main2}
If one of the following conditions holds, then the system~\eqref{eqn:PD4}
 is not meromorphiaclly integrable in the Bogoyavlenskij sense near the $(x_1,x_2)$-plane$:$
\begin{enumerate}
\setlength{\leftskip}{-1.2em}
\item[\rm(i)]
$\alpha_2\alpha_4\beta_5,\alpha_1\alpha_4-\alpha_2\alpha_3,
 \beta_2/\beta_5-\alpha_2/\alpha_4,\beta_6/\beta_5-\alpha_3/\alpha_4\neq 0$,
 $\alpha_2/\alpha_4$ or $\beta_2/\beta_5\not\in\Qset$, and
\begin{enumerate}
\setlength{\leftskip}{0.1em}
\setlength{\labelwidth}{5.6em}
\item[\rm(ia)]
$\alpha_2/\alpha_4$ or $\beta_2/\beta_5-\alpha_2/\alpha_4=1;$
\item[\it or \rm(ib)]
$2\alpha_2/\alpha_4\not\in\Zset_{\ge 2}$,
 $\beta_2/\beta_5-\alpha_2/\alpha_4\not\in\Zset_{\ge 1}$,
 and $\beta_3\beta_5-\beta_2\beta_6\neq 0$ or
\begin{align}
&
(\alpha_1-\alpha_3)(\alpha_2\beta_5-\alpha_4\beta_2)\notag\\
&
 -(\alpha_2-\alpha_4)(\alpha_2\beta_6-\alpha_4\beta_3)
 +(\alpha_1\alpha_4-\alpha_2\alpha_3)\beta_5\neq 0;
\label{eqn:thm2a}
\end{align}
\end{enumerate}
\item[\rm(ii)]
$\alpha_2\alpha_4\beta_5,\alpha_1\alpha_4-\alpha_2\alpha_3,
 \alpha_1\beta_5+\alpha_2\beta_6-\alpha_3\beta_2-\alpha_4\beta_3\neq 0$
 and $\beta_2/\beta_5=\alpha_2/\alpha_4\not\in\Qset;$
\item[\rm(iii)]
$\alpha_2\alpha_4\beta_5,\alpha_1\alpha_4-\alpha_2\alpha_3,
 \beta_2/\beta_5-\alpha_2/\alpha_4\neq 0$,
 $\beta_6/\beta_5=\alpha_3/\alpha_4$,
 $\alpha_2/\alpha_4$ or $\beta_2/\beta_5\not\in\Qset$, and
\begin{enumerate}
\setlength{\leftskip}{0.4em}
\setlength{\labelwidth}{5.6em}
\item[\rm(iiia)]
$\beta_6/\beta_5-\alpha_1/\alpha_2,
 \alpha_1\beta_5-\alpha_2\beta_6+\alpha_3\beta_2-\alpha_4\beta_3\neq 0$, and
\begin{enumerate}
\setlength{\leftskip}{2.1em}
\setlength{\labelwidth}{5.6em}
\item[\rm(iiia1)]
$\alpha_2/\alpha_4=1;$
\item[\it or \rm(iiia2)]
$2\alpha_2/\alpha_4\not\in\Zset_{\ge 2}$ and
\begin{equation}
\alpha_1(\beta_2-\beta_5) -\alpha_2(\beta_3-\beta_6)
 -\alpha_3\beta_6+\alpha_4\beta_3\neq 0;
\label{eqn:thm2b}
\end{equation}
\end{enumerate}
\item[\rm(iiib)]
$\beta_6/\beta_5=\alpha_1/\alpha_2$
 and $\beta_2,\alpha_1\beta_5-\alpha_2\beta_6+\alpha_3\beta_2-\alpha_4\beta_3\neq 0;$
\item[\it or \rm(iiic)]
$\alpha_1\beta_5-\alpha_2\beta_6+\alpha_3\beta_2-\alpha_4\beta_3=0$,
 $\beta_6/\beta_5\neq\alpha_1/\alpha_2$, and
\begin{enumerate}
\setlength{\leftskip}{2.1em}
\setlength{\labelwidth}{5.6em}
\item[\rm(iiic1)]
$\alpha_2/\alpha_4=1$ or $\beta_2/\beta_5-\alpha_2/\alpha_4=-1;$
\item[\it or \rm(iiic2)]
$\beta_2\neq 0$, $2\alpha_2/\alpha_4\not\in\Zset_{\ge 2}$
 and $2(\beta_2/\beta_5-\alpha_2/\alpha_4)\not\in\Zset_{\ge -1};$
\end{enumerate}
\end{enumerate}
\item[\rm(iv)]
$\alpha_2\alpha_4\beta_5,\beta_2/\beta_5-\alpha_2/\alpha_4,\beta_6/\beta_5-\alpha_3/\alpha_4\neq 0$,
 $\alpha_1\alpha_4-\alpha_2\alpha_3=0$,
 $\alpha_2/\alpha_4$ or $\beta_2/\beta_5\not\in\Qset$, and
\begin{enumerate}
\setlength{\leftskip}{0.6em}
\setlength{\labelwidth}{5.6em}
\item[\rm(iva)]
$\alpha_1\beta_5-\alpha_2\beta_6-\alpha_3\beta_2+\alpha_4\beta_3\neq 0$, and
\begin{enumerate}
\setlength{\leftskip}{2.2em}
\setlength{\labelwidth}{5.6em}
\item[\rm(iva1)]
$\beta_2/\beta_5-\alpha_2/\alpha_4=1;$
\item[\it or \rm(iva2)]
$\beta_2/\beta_5-\alpha_2/\alpha_4\not\in\Zset_{\ge 1}$ and
\begin{equation}
(\beta_2-\beta_5)(\alpha_1\beta_5-\alpha_3\beta_2+\alpha_4\beta_3)
 -(\beta_3-\beta_6)\alpha_2\beta_5\neq 0;
\label{eqn:thm2c}
\end{equation}
\end{enumerate}
\item[\it or \rm(ivb)]
$\alpha_1\beta_5-\alpha_2\beta_6-\alpha_3\beta_2+\alpha_4\beta_3=0$, and
\begin{enumerate}
\setlength{\leftskip}{2.4em}
\setlength{\labelwidth}{5.6em}
\item[\rm(ivb1)]
$\alpha_2/\alpha_4=-\tfrac{1}{2}$
 or $\beta_2/\beta_5-\alpha_2/\alpha_4=1;$
\item[\it or \rm(ivb2)]
$2\alpha_2/\alpha_4\not\in\Zset_{\ge -1}$
 and $\beta_2/\beta_5-\alpha_2/\alpha_4\not\in\Zset_{\ge 1};$
\end{enumerate}
\end{enumerate}
\item[\rm(v)]
$\alpha_1\alpha_4\beta_2\beta_5,\beta_6/\beta_5-\alpha_3/\alpha_4\neq 0$,
 $\alpha_2=0$, $\beta_2/\beta_5\not\in\Qset$, and
 $\beta_3\beta_5-\beta_2\beta_6$ or $(\alpha_1-\alpha_3)\beta_2+\alpha_4\beta_3\neq 0;$
\item[\rm(vi)]
$\alpha_4\beta_2\beta_5,\beta_3/\beta_2-\alpha_3/\alpha_4\neq 0$, $\alpha_1,\alpha_2=0$
 and $\beta_2/\beta_5\not\in\Qset;$
\item[\rm(vii)]
$\alpha_2\alpha_3\beta_5\neq 0$, $\alpha_4=0$, and
 $\beta_3\beta_5-\beta_2\beta_6$ or $(\alpha_1-2\alpha_3)\beta_5-\alpha_2\beta_6\neq 0;$
\item[\rm(viii)]
$\alpha_1\alpha_2\beta_5\neq 0$, $\alpha_3,\alpha_4=0$, and
\begin{enumerate}
\setlength{\leftskip}{1.1em}
\setlength{\labelwidth}{5.6em}
\item[\rm(viiia)]
$\beta_6/\beta_5-\alpha_1/\alpha_2,\alpha_1(\beta_2-\beta_5)-\alpha_2(\beta_3-\beta_6)\neq 0;$
\item[\it or \rm(viiib)]
$\beta_6/\beta_5=\alpha_1/\alpha_2$
 and $\beta_2,\beta_3\beta_5-\beta_2\beta_6\neq 0$.
\end{enumerate}
\end{enumerate}
\end{thm}

We give a proof of Theorem~\ref{thm:main2} in Section~5.
We replace $x_1,x_2,x_3,x_4$
 (resp. $\alpha_1,\alpha_2,\alpha_3,\alpha_4$ and $\beta_1,\beta_2,\beta_3,\beta_5,\beta_6$)
 with $x_3,x_4,x_1,x_2$
 (resp. $\alpha_3,\alpha_4,\alpha_2,\alpha_1$ and\linebreak $\beta_5,\beta_4,\beta_6,\beta_1,\beta_3$)
 in Theorem~\ref{thm:main2}, and immediately obtain the following.

\begin{thm}
\label{thm:main3}
If one of the following conditions holds, then the system~\eqref{eqn:PD4}
 is not meromorphiaclly integrable in the Bogoyavlenskij sense near the $(x_3,x_4)$-plane$:$
\begin{enumerate}
\setlength{\leftskip}{-1.2em}
\item[\rm(i)]
$\alpha_1\alpha_3\beta_1,\alpha_1\alpha_4-\alpha_2\alpha_3,
 \beta_4/\beta_1-\alpha_3/\alpha_1,\beta_3/\beta_1-\alpha_2/\alpha_1\neq 0$,
 $\alpha_3/\alpha_1$ or $\beta_4/\beta_1\not\in\Qset$, and
\begin{enumerate}
\setlength{\leftskip}{0.1em}
\setlength{\labelwidth}{5.6em}
\item[\rm(ia)]
$\alpha_3/\alpha_1=1$ or $\beta_4/\beta_1-\alpha_3/\alpha_1=1;$
\item[\it or \rm(ib)]
$2\alpha_3/\alpha_1\not\in\Zset_{\ge 2}$,
 $\beta_4/\beta_1-\alpha_3/\alpha_1\not\in\Zset_{\ge 1}$,
 and $\beta_1\beta_6-\beta_3\beta_4\neq 0$ or
\begin{align*}
&
(\alpha_4-\alpha_2)(\alpha_3\beta_1-\alpha_1\beta_4)\notag\\
&
 -(\alpha_3-\alpha_1)(\alpha_3\beta_3-\alpha_1\beta_6)
 +(\alpha_1\alpha_4-\alpha_2\alpha_3)\beta_1\neq 0;
\end{align*}
\end{enumerate}
\item[\rm(ii)]
$\alpha_1\alpha_3\beta_1,\alpha_1\alpha_4-\alpha_2\alpha_3,
 \alpha_4\beta_1+\alpha_3\beta_3-\alpha_2\beta_4-\alpha_1\beta_6\neq 0$
 and $\alpha_3/\alpha_1=\beta_4/\beta_1\not\in\Qset;$
\item[\rm(iii)]
$\alpha_1\alpha_3\beta_1,\alpha_1\alpha_4-\alpha_2\alpha_3,
 \beta_4/\beta_1-\alpha_3/\alpha_1\neq 0$,
 $\beta_3/\beta_1=\alpha_2/\alpha_1$,
 $\alpha_3/\alpha_1$ or $\beta_4/\beta_1\not\in\Qset$, and
\begin{enumerate}
\setlength{\leftskip}{0.4em}
\setlength{\labelwidth}{5.6em}
\item[\rm(iiia)]
$\beta_3/\beta_1-\alpha_4/\alpha_3,
 \alpha_4\beta_1-\alpha_3\beta_3+\alpha_2\beta_4-\alpha_1\beta_6\neq 0$, and
\begin{enumerate}
\setlength{\leftskip}{2.1em}
\setlength{\labelwidth}{5.6em}
\item[\rm(iiia1)]
$\alpha_3/\alpha_1=1;$
\item[\it or \rm(iiia2)]
$2\alpha_3/\alpha_1\not\in\Zset_{\ge 2}$ and
\[
\alpha_4(\beta_4-\beta_1) -\alpha_3(\beta_6-\beta_3)
 -\alpha_2\beta_3+\alpha_1\beta_6\neq 0;
\]
\end{enumerate}
\item[\rm(iiib)]
$\beta_3/\beta_1=\alpha_4/\alpha_3$
 and $\beta_4,\alpha_4\beta_1-\alpha_3\beta_3+\alpha_2\beta_4-\alpha_1\beta_6\neq 0;$
\item[\it or \rm(iiic)]
$\alpha_4\beta_1-\alpha_3\beta_3+\alpha_2\beta_4-\alpha_1\beta_6=0$,
 $\beta_3/\beta_1\neq\alpha_4/\alpha_3$, and
\begin{enumerate}
\setlength{\leftskip}{2.1em}
\setlength{\labelwidth}{5.6em}
\item[\rm(iiic1)]
$\alpha_3/\alpha_1=1$ or $\beta_4/\beta_1-\alpha_3/\alpha_1=-1;$
\item[\it or \rm(iiic2)]
$\beta_4\neq 0$, $2\alpha_3/\alpha_1\not\in\Zset_{\ge 2}$
 and $2(\beta_4/\beta_1-\alpha_3/\alpha_1)\not\in\Zset_{\ge -1};$
\end{enumerate}
\end{enumerate}
\item[\rm(iv)]
$\alpha_1\alpha_3\beta_1,\beta_4/\beta_1-\alpha_3/\alpha_1,\beta_3/\beta_1-\alpha_2/\alpha_1\neq 0$,
 $\alpha_1\alpha_4-\alpha_2\alpha_3=0$,
 $\alpha_3/\alpha_1$ or $\beta_4/\beta_1\not\in\Qset$, and
\begin{enumerate}
\setlength{\leftskip}{0.6em}
\setlength{\labelwidth}{5.6em}
\item[\rm(iva)]
$\alpha_4\beta_1-\alpha_3\beta_3-\alpha_2\beta_4+\alpha_1\beta_6\neq 0$, and
\begin{enumerate}
\setlength{\leftskip}{2.2em}
\setlength{\labelwidth}{5.6em}
\item[\rm(iva1)]
$\beta_4/\beta_1-\alpha_3/\alpha_1=1;$
\item[\it or \rm(iva2)]
$\beta_4/\beta_1-\alpha_3/\alpha_1\not\in\Zset_{\ge 1}$ and
\[
(\beta_4-\beta_1)(\alpha_4\beta_1-\alpha_2\beta_4+\alpha_1\beta_6)
 -(\beta_6-\beta_3)\alpha_3\beta_1\neq 0;
\]
\end{enumerate}
\item[\it or \rm(ivb)]
$\alpha_4\beta_1-\alpha_3\beta_3-\alpha_2\beta_4+\alpha_1\beta_6=0$, and
\begin{enumerate}
\setlength{\leftskip}{2.4em}
\setlength{\labelwidth}{5.6em}
\item[\rm(ivb1)]
$\alpha_3/\alpha_1=-\tfrac{1}{2}$
 or $\beta_4/\beta_1-\alpha_3/\alpha_1=1;$
\item[\it or \rm(ivb2)]
$2\alpha_3/\alpha_1\not\in\Zset_{\ge -1}$
 and $\beta_4/\beta_1-\alpha_3/\alpha_1\not\in\Zset_{\ge 1};$
\end{enumerate}
\end{enumerate}
\item[\rm(v)]
$\alpha_1\alpha_4\beta_1\beta_4,\beta_3/\beta_1-\alpha_2/\alpha_1\neq 0$,
 $\alpha_3=0$, $\beta_4/\beta_1\not\in\Qset$, and
 $\beta_1\beta_6-\beta_3\beta_4$ or $(\alpha_4-\alpha_2)\beta_4+\alpha_1\beta_6\neq 0;$
\item[\rm(vi)]
$\alpha_1\beta_1\beta_4,\beta_6/\beta_4-\alpha_2/\alpha_1\neq 0$,
 and $\alpha_3,\alpha_4=0$ and $\beta_4/\beta_1\not\in\Qset;$
\item[\rm(vii)]
$\alpha_3\alpha_2\beta_1\neq 0$, $\alpha_1=0$, and
 $\beta_1\beta_6-\beta_3\beta_4$ or $(\alpha_4-2\alpha_2)\beta_1-\alpha_3\beta_3\neq 0;$
\item[\rm(viii)]
$\alpha_4\alpha_3\beta_1\neq 0$, $\alpha_1,\alpha_3=0$, and
\begin{enumerate}
\setlength{\leftskip}{1.1em}
\setlength{\labelwidth}{5.6em}
\item[\rm(viiia)]
$\beta_3/\beta_1-\alpha_4/\alpha_3,\alpha_4(\beta_4-\beta_1)-\alpha_3(\beta_6-\beta_3)\neq 0;$
\item[\it or \rm(viiib)]
$\beta_3/\beta_1=\alpha_4/\alpha_2$
 and $\beta_4,\beta_3\beta_5-\beta_4\beta_6\neq 0$.
\end{enumerate}
\end{enumerate}
\end{thm}

\begin{rmk}
Assume that the hypotheses of Theorem~$\ref{thm:Y2}$ do  not hold, i.e.,
\begin{enumerate}
\setlength{\leftskip}{-0.6em}
\item[\rm(i)]
$\alpha_1=\alpha_3$ or $\alpha_2=\alpha_4$;
\item[\rm(ii)]
$\alpha_1\alpha_3,\alpha_2\alpha_4\le 0$,
 and $\alpha_1,\alpha_2,\alpha_3$ or $\alpha_4=0;$
\item[or \rm(iii)]
$\alpha_1\alpha_3,\alpha_2\alpha_4<0$,
 and $\alpha_1/\alpha_3$ or $\alpha_2/\alpha_4\in\Qset$.
\end{enumerate}
In both cases {\rm(i)} and {\rm(ii)},
 conditions~{\rm(i)}-{\rm(iii)} and {\rm(v)}-{\rm(viii)}
 of Theorems~$\ref{thm:main2}$ and $\ref{thm:main3}$ may hold.
In case {\rm(iii)},
 conditions~{\rm(i)}-{\rm(iv)} of Theorems~$\ref{thm:main2}$ and $\ref{thm:main3}$ may hold.
Thus, even if the hypotheses of Theorem~$\ref{thm:Y2}$ do not hold,
 the system~\eqref{eqn:PD3} may be meromorphically nonintegrable.
So Theorems~$\ref{thm:main2}$ and $\ref{thm:main3}$ give answers
 to the questions stated above for \eqref{eqn:PD4}.
\end{rmk}

The outline of this paper is as follows:
In Section~2 we reduce the nonintegrability of \eqref{eqn:PD3} and \eqref{eqn:PD4}
  to that of simple planar systems.
For this purpose, we use Proposition~2.1 of \cite{AY20},
 which enables us to make such a reduction for a class of systems
 including \eqref{eqn:PD3} and \eqref{eqn:PD4}.
In Section~3 we collect necessary results of \cite{AY20} for planar vector fields
 and give a novel one on scalar first-order linear differential equations,
 which is extensively used in the remaining sections.
In Sections~4 and 5, we prove Theorems~\ref{thm:main1} and \ref{thm:main2}, respectively,
 using the results of Sections~2 and 3.


\section{Reduction to Planar Systems}

Now we consider \eqref{eqn:PD3} and \eqref{eqn:PD4}.
We first use the change of coordinate $(x_1,x_2)=(r\cos\theta,r\sin\theta)$
 to transform \eqref{eqn:PD3} to
\begin{equation}
\begin{split}
&
\dot r=(\alpha_1x_3+\beta_1r^2+\beta_2x_3^2)r,\\
&
\dot x_3=\alpha_3r^2+\alpha_4x_3^2+(\beta_5r^2+\beta_6x_3^2)x_3,\\
&
\dot\theta=\omega+(\alpha_2+\beta_3r^2+\beta_4x_3^2)x_3,
\end{split}
\label{eqn:PD3p}
\end{equation}
of which the $(r,x_3)$-components are independent of $\theta$.
Using the change of coordinates $(x_1,x_2)=(r_1\cos\theta_1,r_1\sin\theta_1)$
 and $(x_3,x_4)=(r_2\cos\theta_2,r_2\sin\theta_2)$,
 we also transform \eqref{eqn:PD4} to
\begin{equation}
\begin{split}
&
\dot{r}_1
= r_1(\alpha_1 r_1^2+\alpha_2r_2^2+\beta_1r_1^4+\beta_2r_2^4+\beta_3 r_1^2r_2^2),\\
&
\dot{r}_2
= r_2(\alpha_3 r_1^2+\alpha_4r_2^2+\beta_4r_1^4+\beta_5r_2^4+\beta_6 r_1^2r_2^2),\\
&
\dot{\theta}_1
=\omega_1+(\tilde{\alpha}_1 r_1^2+\tilde{\alpha}_2r_2^2
 +\tilde{\beta}_1r_1^4+\tilde{\beta}_2r_2^4+\tilde{\beta}_3 r_1^2r_2^2),\\
&
\dot{\theta}_2
=\omega_2+(\tilde{\alpha}_3 r_1^2+\tilde{\alpha}_4r_2^2
 +\tilde{\beta}_4r_1^4+\tilde{\beta}_5r_2^4+\tilde{\beta}_6 r_1^2r_2^2),
\end{split}
\label{eqn:PD4p}
\end{equation}
of which the  $(r_1,r_2)$-components are independent of $\theta_1$ and $\theta_2$.
So we expect that one can reduce the nonintegrability of \eqref{eqn:PD3} and \eqref{eqn:PD4}
 to that of the $(r,x_3)$-components of \eqref{eqn:PD3p},
\begin{equation}
\dot r=(\alpha_1x_3+\beta_1r^2+\beta_2x_3^2)r,\quad
\dot x_3=\alpha_3r^2+\alpha_4x_3^2+(\beta_5r^2+\beta_6x_3^2)x_3,
\label{eqn:PD3p0}
\end{equation}
and the $(r_1,r_2)$-components of \eqref{eqn:PD4p},
\begin{equation}
\begin{split}
&
\dot{r}_1
= r_1(\alpha_1 r_1^2+\alpha_2r_2^2+\beta_1r_1^4+\beta_2r_2^4+\beta_3r_1^2r_2^2),\\
&
\dot{r}_2
= r_2(\alpha_3 r_1^2+\alpha_4r_2^2+\beta_4r_1^4+\beta_5r_2^4+\beta_6r_1^2r_2^2),
\end{split}
\label{eqn:PD4p0}
\end{equation}
respectively.
This is true as follows.

Let $m>0$ be an integer and consider $m+2$-dimensional systems of the form
\begin{equation}
\dot{x}=f_x(x,y),\quad
\dot{y}=f_y(x,y),\quad
(x,y)\in D,
\label{eqn:fg}
\end{equation}
where $D\subset\Cset^2\times\Cset^m$ is a region containing $m$-dimensional $y$-plane
 $\{(0,y)\in\Cset^2\times\Cset^m\mid y\in\Cset^m\}$,
 and $f_x:D\to\Cset^2$ and $f_y:D\to\Cset^m$ are analytic.
Assume that by the change of coordinates
 $x=(x_1,x_2)=(r\cos\theta,r\sin\theta)$,
 Eq.~\eqref{eqn:fg} is transformed to
\begin{equation}
\dot{r}=R(r,y),\quad
\dot{y}=\tilde{f}_y(r,y),\quad
\dot{\theta}=\Theta(r,y),\quad
(r,y,\theta)\in\tilde{D}\times\Cset,
\label{eqn:Rg0}
\end{equation}
where $\tilde{D}\subset\Cset\times\Cset^m$ is a region
 containing the $m$-dimensional $y$-plane
 $\{(0,y)\in\Cset\times\Cset^m\mid y\in\Cset^m$\},
 and $R:\tilde{D}\to\Cset$, $\tilde{f}_y:\tilde{D}\to\Cset^m$
 and $\Theta:\tilde{D}\to\Rset$ are analytic.
Note that $\tilde{f}_y(r,y)=f_y(r\cos\theta,r\sin\theta,y)$.
We are especially interested in the $(r,y)$-components of \eqref{eqn:Rg0},
\begin{equation}
\dot{r}=R(r,y),\quad
\dot{y}=\tilde{f}_y(r,y),
\label{eqn:Rg}
\end{equation}
which are independent of $\theta$.
In this situation we have the following proposition.

\begin{prop}
\label{prop:2a}\
\begin{enumerate}
\setlength{\leftskip}{-1.8em}
\item[\rm(i)]
Suppose that Eq.~\eqref{eqn:fg} has a meromorphic first integral $F(x_1,x_2,y)$ near $(x_1,x_2)=(0,0)$,
 and let $\tilde{F}(r,\theta,y)=F(r\cos\theta,r\sin\theta,y)$.
If $\tilde{f}_{yj}(0,y)\neq 0$ for almost all $y\in\Cset^m$ for some $j=1,\ldots,m$, then
\[
G(r,y)=\tilde{F}(r,\tilde{\theta}_j(y_j),y)
\]
is a meromorphic first integral of \eqref{eqn:Rg} near $r=0$,
 where $y_j$ and $\tilde{f}_{yj}(r,y)$ are the $j$-th components of $y$ and $\tilde{f}_y(r,y)$, respectively,
 and $\tilde{\theta}_j(y_j)$ represents the $\theta$-component of a solution to
\begin{equation}
\frac{\d r}{\d y_j}=\frac{R(r,y)}{\tilde{f}_{yj}(r,y)},\quad
\frac{\d y_\ell}{\d y_j}=\frac{\tilde{f}_{y\ell}(r,y)}{\tilde{f}_{yj}(r,y)},\quad
\frac{\d\theta}{\d y_j}=\frac{\Theta(r,y)}{\tilde{f}_{yj}(r,y)},\quad\ell\neq j.
\label{eqn:prop2a1}
\end{equation}
\item[\rm(ii)]
Suppose that Eq.~\eqref{eqn:fg} has a  meromorphic commutative vector field
\begin{equation}
v(x_1,x_2,y):=
\begin{pmatrix}
v_1(x_1,x_2,y)\\
v_2(x_1,x_2,y)\\
v_y(x_1,x_2,y)
\end{pmatrix}
\label{eqn:prop2a2}
\end{equation}
with $v_1,v_2:D\to\Cset$ and $v_y:D\to\Cset^m$ near $(x_1,x_2)=(0,0)$.
If $\Theta(0,y)\neq 0$ for almost all $y\in\Cset^m$, then
\begin{align}
&
\begin{pmatrix}
\tilde{v}_r(r,\theta,y)\\
\tilde{v}_y(r,\theta,y)
\end{pmatrix}\notag\\
&=
\begin{pmatrix}
v_1(r\cos\theta,r\sin\theta,y)\cos\theta+v_2(r\cos\theta,r\sin\theta,y)\sin\theta\\
v_y(r\cos\theta,r\sin\theta,y)
\end{pmatrix}
\label{eqn:prop2a3}
\end{align}
is independent of $\theta$
 and it is a  meromorphic commutative vector field of \eqref{eqn:Rg} near $r=0$.
\end{enumerate}
\end{prop}

See Proposition~2.1 of \cite{AY20} for the proof.
Using Proposition~\ref{prop:2a} for \eqref{eqn:PD3} and \eqref{eqn:PD4}
 (once for the former and twice for the latter),
 we immediately obtain the following corollaries.

\begin{cor}
\label{cor:PD3}
If the complexification of \eqref{eqn:PD3} is meromorphically integrable
 near $(x_1,x_2)=(0,0)$, then so is Eq.~\eqref{eqn:PD3p0} near $r=0$.
\end{cor}

\begin{cor}
\label{cor:PD4}
If the complexification of \eqref{eqn:PD4} is meromorphically integrable
 near $(x_1,x_2)=(0,0)$ and near $(x_3,x_4)=(0,0)$,
 then so is Eq.~\eqref{eqn:PD4p0} near $r_1=0$ and near $r_2=0$, respectively.
\end{cor}


\section{Nonintegrability of Planar Polynomial Vector Fields}

We next collect the necessary results of \cite{AY20}
 for planar polynomial vector fields,
 which are used for \eqref{eqn:PD3p0} and \eqref{eqn:PD4p0} in Sections~4 and 5, respectively,
 where some errors found in \cite{AY20} are corrected.
We also give a useful result on scalar first-order linear differential equations,
 which is extensively used in Sections~4 and 5.

Consider planar polynomial vector fields of the form
\begin{equation}\label{eqn:pls}
\dot{\xi}=P(\xi,\eta),\quad
\dot{\eta}=Q(\xi,\eta),\quad
(\xi,\eta)\in\Cset^2,
\end{equation}
where $P(\xi,\eta)$ and $Q(\xi,\eta)$ are polynomials.
Let $\Gamma: \eta-\varphi(\xi)=0$ be an integral curve of \eqref{eqn:pls},
 where $\varphi(\xi)$ is assumed to be a rational function of $\xi$.
So $\Gamma$ represents a rational solution to the first-order differential equation
\begin{equation} \label{eqn:fol}
\eta^\prime=\frac{Q(\xi,\eta)}{P(\xi,\eta)}=:R(\xi,\eta),
\end{equation}
which defines a foliation associated with \eqref{eqn:pls} (or its orbits),
 where the prime denotes differentiation with respect to $\xi$
 and $R(\xi,\eta)$ is rational in $\xi$ and $\eta$.

Let $\phi(\xi,\eta)$ denote the (nonautonomous) flow
 of the one-dimensional system \eqref{eqn:fol} with $\phi(\xi_0,\eta)=\eta$ for $\xi_0$ fixed,
 and let $(\xi_0,\eta_0)$ be a point on $\Gamma$, i.e., $\eta_0=\varphi(\xi_0)$.
We are interested in the variation of $\phi(\xi,\eta)$
 with respect to $\eta$ around $\eta=\eta_0$ at $\xi=\xi_0$, which is expressed as
\[
\phi(\xi,\eta)=\varphi(\xi)+\frac {\partial \phi}{\partial \eta}(\xi,\eta_0)(\eta-\eta_0)
+\frac 12\frac {\partial^2 \phi}{\partial \eta^2}(\xi,\eta_0)(\eta-\eta_0)^2+\cdots.
\]
So we want to compute the above Taylor expansion coefficients
\[
\varphi_k(\xi)=\frac {\partial^k \phi}{\partial \eta^k}(\xi,\eta_0),\quad
k\in\Nset,
\]
which are solutions to the equations in variation. 
Let
\begin{equation}
\kappa_k(\xi):=\frac{\partial^k R}{\partial \eta^k}(\xi,\varphi(\xi)),\quad
k\in\Nset.
\label{eqn:kappa}
\end{equation}
Note that $\kappa_k(\xi)$ is rational for any $k\in\Nset$.
{\renewcommand{\theequation}{$\mathrm{VE}_1$}
The \emph{first-} and \emph{second-order variational equations}
 ($\mathrm{VE}_1$) and ($\mathrm{VE}_2$) are given by
\begin{equation}\label{VE1}
\varphi_1'=\kappa_1(\xi)\varphi_1
\end{equation}
}and
{\renewcommand{\theequation}{$\mathrm{VE}_2$}
\begin{equation}\label{VE2}
 \varphi_1^\prime=\kappa_1(\xi)\varphi_1,\quad
 \varphi_2^\prime=\kappa_1(\xi)\varphi_2+ \kappa_2(\xi)\varphi_1^2,
\end{equation}
}respectively.
The $\mathrm{VE}_1$ is linear
 but the $\mathrm{VE}_2$ is nonlinear.
Letting $\chi_{21}:=\varphi_1^2$ and $\chi_{22}:=\varphi_2$,
 we can linearize the $\mathrm{VE}_2$ as
{\renewcommand{\theequation}{$\mathrm{LVE}_2$}
\begin{equation}\label{LVE2}
 \chi_{21}^\prime=2\kappa_1(\xi)\chi_{21},\quad
 \chi_{22}^\prime=\kappa_1(\xi)\chi_{22}+\kappa_2(\xi)\chi_{21},
\end{equation}
}and refer to it as the \emph{second-order linearized variational equation}
 ($\mathrm{LVE}_2$).
We also refer to the $\mathrm{VE}_1$ as the $\mathrm{LVE}_1$. 
In a similar way, for any $k>2$,
 we obtain the \emph{$k$th-order variational equation} $\mathrm{VE}_k$ as
{\renewcommand{\theequation}{$\mathrm{VE}_k$}
\begin{align}\label{VEk}
&
\varphi'_1=\kappa_1(\xi)\varphi_1,\quad
\varphi'_2=\kappa_1(\xi)\varphi_2+\kappa_2(\xi)\varphi_1^2,
\quad\ldots,\notag\\
&
\varphi'_k=\kappa_1(\xi)\varphi_k
 +\cdots+\tfrac{1}{2}k(k-1)\kappa_{k-1}(\xi)\varphi_1^{k-2}\varphi_2 +\kappa_k(\xi)\varphi_1^k.
\end{align}
}We can also linearize the $\mathrm{VE}_k$ as
{\renewcommand{\theequation}{$\mathrm{LVE}_k$}
\begin{align}\label{LVEk}
&
\chi_{k1}^\prime=k\kappa_1(\xi)\chi_{k1},\quad
\chi_{k2}^\prime=(k-1)\kappa_1(\xi)\chi_{k2}+\kappa_2(\xi)\chi_{k1},\quad\ldots,\notag\\
&
\chi_{kk}^\prime=\kappa_1(\xi)\chi_{kk}+\cdots+\tfrac{1}{2}k(k-1)\kappa_{k-1}(\xi)\chi_{k2}+\kappa_k(\xi)\chi_{k1},
\end{align}
}and refer to it as the \emph{$k$th-order linearized variational equation}
 ($\mathrm{LVE}_k$),
 where $\chi_{k1}=\varphi_1^k,\chi_{k2}=\varphi_1^{k-2}\varphi_2,\ldots,\chi_{kk}=\varphi_k$.
\setcounter{equation}{3}
We observe that the $\mathrm{LVE}_k$ has a two-dimensional subsystem
\begin{equation}\label{eqn:sLVEk}
\chi_{k1}'=k \kappa_1(\xi)\chi_{k1},\quad
\chi_{kk}'=\kappa_1(\xi)\chi_{kk} +\kappa_k(\xi)\chi_{k1}
\end{equation}
for any $k\ge 2$.

Let $\G_k$ be the differential Galois group of the $\mathrm{LVE}_k$
 and let $\G_k^0$ be its identity component.
Using the result of Ayoul and Zung \cite{AZ10}
 based on the seminal results of Morales-Ruiz, Ramis and Sim\'o \cite{M99,MR01,MRS07},
 we have the following theorem  \cite{ALMP18}.

\begin{thm}
\label{thm:a}
Suppose that the $\mathrm{VE}_1$ has no irregular singularity at infinity
 and the planar polynomial vector field \eqref{eqn:pls}
 is meromorphically integrable in a neighbourhood of $\Gamma$.
Then for any $k\geq 1$ the identity component $\G_k^0$ is commutative.
\end{thm}

The statement of Theorem~\ref{thm:a} also holds in a more general setting.
See \cite{AZ10,M99,MR01,MRS07} for the details.
Obviously, $\G_1$ and $\G_1^0$ are subgroups of $\mathbb{C}^\ast$ and commutative.
However, $\G_k$ and $\G_k^0$ may not be commutative for $k\ge 2$.

Let
\begin{equation}
\Omega(\xi)=\exp\left(\int\kappa_1(\xi)\,\d\xi\right),\quad
\Theta_k(\xi)=\int\kappa_k(\xi)\Omega(\xi)^{k-1}\d\xi
\label{eqn:omega}
\end{equation}
for $k\ge 2$.
The subsystem \eqref{eqn:sLVEk} of the $\mathrm{LVE}_k$
 has two linearly independent solutions
 $(\chi_{k1},\chi_{kk})=(0,\Omega(\xi))$ and $(\Omega(\xi)^k,\Omega(\xi)\Theta_k(\xi))$.
Let $\tilde{\G}$ be the differential Galois group of \eqref{eqn:sLVEk}
 and let $\tilde{\G}^0$ be its identity component.

Let $\kappa_k(\xi)=\kappa_{k\n}(\xi)/\kappa_{k\d}(\xi)$ for $k\in\Nset$,
 where $\kappa_{k\n}(\xi)$ and $\kappa_{k\d}(\xi)$ are relatively prime polynomials
 and $\kappa_{k\d}(\xi)$ is monic.
We see that if $\deg(\kappa_{1\d})>\deg(\kappa_{1\n})$,
 then $\kappa_1(1/\xi)/\xi$ is holomorphic at $\xi=0$,
 so that the $\mathrm{VE}_1$ and consequently the $\mathrm{LVE}_k$
 have no irregular singularity at infinity for $k\ge 2$.
Using Theorem~\ref{thm:a}, we can prove the following theorem
 (see Theorem~3.3 of \cite{AY20} for the proof).

\begin{thm}
\label{thm:c}
Suppose that $\deg(\kappa_{1\d})>\deg(\kappa_{1\n})$
 and the following conditions hold for some $k\ge 2:$
\begin{enumerate}
\setlength{\leftskip}{-1em}
\item[\rm(H1)]
$\Omega(\xi)$ is transcendental\,$;$
\item[\rm(H2)]
$\Theta_k(\xi)/\Omega(\xi)^{k-1}$ is not rational.
\end{enumerate}
Then the planar polynomial vector field \eqref{eqn:pls}
 is meromorphically nonintegrable in a neighbourhood of $\Gamma$.
\end{thm}

It is often difficult to check condition~{\rm(H2)} directly
 in application of Theorem~\ref{thm:c}.
We give useful criteria for condition (H2).
They are extensively used in our proofs of the main theorems in Sections~4 and 5.

For $k\ge 2$, we write
\begin{equation}
\kappa_{k\d}(\xi)
 =\kappa_{1\d}(\xi)\prod_{j=1}^{n_1}(\xi-\xi_{1j})^{a_{1j}}
 \prod_{j=1}^{n_k}(\xi-\xi_{kj})^{a_{kj}},
\label{eqn:kd}
\end{equation}
where $n_\ell\in\Zset_{\ge 0}$,
 $\xi_{\ell j}\in\Cset$ and $a_{\ell j}\in\Zset\setminus\{0\}$,
 $j=1,\ldots,n_\ell$, if $n_\ell>0$ for $\ell=1,k$,
 such that $\xi_{1j}$ is a root of $\kappa_{1\d}(\xi)$ but $\xi_{kj}$ is not,
 and $\xi_{\ell j_1}\neq\xi_{\ell j_2}$ if $j_1\neq j_2$.
Let $b_{1j}$ denote the multiplicity of the zero $\xi_{1j}$ for $\kappa_{1\d}(\xi)$.
Note that $a_{kj}>0$, $j=1,\ldots,n_k$, if $n_k>0$
 and that $a_{1j}\ge -b_{1j}$ but $a_{1j}\neq 0$, $j=1,\ldots,n_1$, if $n_1>0$,
 since $\kappa_{k\d}(\xi)$ is a polynomial and $\kappa_{1\d}(\xi_{kj})\not\equiv 0$, $j=1,\ldots,n_k$.
When $n_1>0$, let
\begin{equation}
\bar{\kappa}_{kb}(\xi)=(k-1)\kappa_{1\n}(\xi)\prod_{j=1}^{n_1}(\xi-\xi_{1j})
 -\kappa_{1\d}(\xi)\sum_{j=1}^{n_1}(a_{1j}+b_j-1)\prod_{\ell\neq j}(\xi-\xi_{1\ell})
\label{eqn:bkappa}
\end{equation}
for $b=(b_1,\ldots,b_{n_1})$ with $b_j\in\Nset$, $j=1,\ldots.n_1$.
Obviously, $\bar{\kappa}_{kb}(\xi)$ has a zero at $\xi=\xi_{1j}$ 
 by $\kappa_{1\d}(\xi_{1j})=0$.
We see that if $b_{1j}>1$, i.e., the zero $\xi_{1j}$ is not simple for $\kappa_{1\d}(\xi)$,
 then it is simple for $\bar{\kappa}_{kb}(\xi)$ with any $b\in\Nset^{n_1}$,
 since  $\kappa_{1\n}(\xi_{1j})\neq 0$.
We define the polynomial
\begin{align}
\rho_k(\xi)
=& 
 (k-1)\kappa_{1\n}(\xi)\prod_{j=1}^{n_k}(\xi-\xi_{kj})
 -\kappa_{1\d}(\xi)\sum_{j=1}^{n_k}(a_{kj}-1)\prod_{\ell\neq j}(\xi-\xi_{k\ell})\notag\\
& -\frac{\displaystyle\kappa_{1\d}(\xi)\prod_{j=1}^{n_k}(\xi-\xi_{kj})}
 {\displaystyle\prod_{j=1}^{n_1}(\xi-\xi_{1j})}
  \sum_{j=1}^{n_1}a_{1j}\prod_{\ell\neq j}(\xi-\xi_{1\ell}),
\label{eqn:rhok}
\end{align}
where a typographical error in the last term of (3.14) in \cite{AY20} has been fixed.
Let $\bar{\rho}_k(\xi)$ and $\tilde{\rho}_k(\xi)$ be the quotient and remainder,
 respectively, when $\kappa_{k\n}(\xi)$ is divided by $\rho_k(\xi)$.
So $\kappa_{k\n}(\xi)=\bar{\rho}_k(\xi)\rho_k(\xi)+\tilde{\rho}_k(\xi)$.
Let $\bar{n}\in\Zset_{\ge 0}$ be the number of distinct roots of $\bar{\rho}_k(\xi)$.
We also consider the first-order linear differential equation
\begin{equation}
\left(\kappa_{1\d}(\xi)\prod_{j=1}^{n_k}(\xi-\xi_{kj})\right)z'+\rho_k(\xi)z
 =\kappa_{k\n}(\xi),
\label{eqn:prop3a}
\end{equation}
where the prime represents differentiation with respect to $\xi$.
Let $\rho_{k0}$ be the leading coefficient of $\rho_k(\xi)$
 and let $\bar{\kappa}_k(\xi)=\bar{\kappa}_{kb}(\xi)$
 with $b=(1,\ldots,1)\in\Nset^{n_1}$, i.e.,
\[
\bar{\kappa}_k(\xi)=(k-1)\kappa_{1\n}(\xi)\prod_{j=1}^{n_1}(\xi-\xi_{1j})
 -\kappa_{1\d}(\xi)\sum_{j=1}^{n_1}a_{1j}\prod_{\ell\neq j}(\xi-\xi_{1\ell}).
\]
We can prove the following proposition.

\begin{prop}
\label{prop:3a}
Let $k\ge 2$.
Suppose that $\kappa_{1\n}(\xi),\kappa_{k\n}(\xi)\not\equiv 0$
 and $\kappa_{k\n}(\xi_{1j})\neq 0$, $j=1,\ldots,n_1$.
If one of the following conditions holds, then condition {\rm(H2)} holds$:$
\begin{enumerate}
\setlength{\leftskip}{-1.8em}
\item[\rm(i)]
$a_{kj}=1$ for some $j=1,\ldots,n_k;$
\item[\rm(ii)]
For each $j\in\{1,\ldots,n_1\}$ the zero $\xi_{1j}$ is not simple for $\bar{\kappa}_k(\xi)$
 or simple for $\bar{\kappa}_{kb}(\xi)$ with some $b\in\Nset^{n_1}$ for each $b_j\in\Nset$.
\end{enumerate}
Moreover, if $n_1>0$, then assume that for $j=1,\ldots,n_1$
 the zero $\xi_{1j}$ of $\bar{\kappa}_{kb}(\xi)$ is simple
 when $b_j>1$.
If one of the following conditions holds, then condition {\rm(H2)} holds$:$
\begin{enumerate}
\setlength{\leftskip}{-1.4em}
\item[\rm(iii)]
Eq.~\eqref{eqn:prop3a} does not have a polynomial solution
 that has no root at $\xi=\xi_{1j}$ and $\xi_{k\ell}$
 for any $j=1,\ldots,n_1$ and $\ell=1,\ldots,n_k;$
\item[\rm(iv)]
$\bar{n}=0$,
\begin{enumerate}
\item[\rm(iva)]
$\bar{\rho}_k(\xi)\equiv 0$ or $\tilde{\rho}_k(\xi)\not\equiv 0;$
\item[and \rm(ivb)]
$\deg(\kappa_{1\d})+n_k\neq\deg(\rho_k)+1$ or $-\rho_{k0}\notin\Nset;$
\end{enumerate}
\item[\rm(v)]
$\bar{n}>0$ and $\deg(\kappa_{1\d})+n_k>\max(\deg(\kappa_{k\n}),\deg(\rho_k)+1);$
\item[\rm(vi)]
$\bar{n}>0$, $\deg(\kappa_{1\d})+n_k<\deg(\rho_k)-\deg(\bar{\rho}_k)+1$ and
\begin{enumerate}
\item[\rm(via)]
$\bar{\rho}_k(\xi)$ has a root at $\xi=\xi_{1j}$ or $\xi_{k\ell}$
 for some $j=1,\ldots,n_1$ or $\ell=1,\ldots,n_k;$
\item[or \rm(vib)]
$\displaystyle
\tilde{\rho}_k(\xi)\not\equiv\kappa_{1\d}(\xi)\bar{\rho}_k'(\xi)\prod_{j=1}^{n_k}(\xi-\xi_{kj})$.
\end{enumerate}
\end{enumerate}
\end{prop}

See Section~3.2 of \cite{AY20} for the proof.
In application of Proposition~\ref{prop:3a} in Sections~4 and 5,
 condition~(iii) is slightly strengthened and replaced with one that
 Eq.~\eqref{eqn:prop3a} simply has no polynomial solution.
 
Let $a(\xi),b(\xi),c(\xi)$ be nonzero polynomials of $\xi$
 and consider the first-order linear differential equation
\begin{equation}
a(\xi)z'+b(\xi)z=c(\xi).
\label{eqn:prop3b}
\end{equation}
Let $\ell_a=\deg(a)$, $\ell_b=\deg(b)$ and $\ell_c=\deg(c)$,
 and let $a_0$ and $b_0$ denote the leading coefficients of $a(\xi)$ and $b(\xi)$, respectively.
Let $\Zset_{\le l}=\{j\in\Zset\mid j\le l\}$.
The following proposition is useful to check condition~(iii) in Proposition~\ref{prop:3a},
 and extensively used in Sections~4 and 5.

\begin{prop}
\label{prop:3b}
Suppose that $a_0=1$, $c(\xi)\not\equiv 0$, $b(\xi)/c(\xi)\not\in\Cset$
 and Eq.~\eqref{eqn:prop3b} has a polynomial solution.
Then the following hold$\,:$
\begin{enumerate}
\setlength{\leftskip}{-1.6em}
\item[\rm(i)]
$\ell_c\ge\max(\ell_a,\ell_b+1)$ unless $\ell_a=\ell_b+1$ and $b_0\in\Zset_{\le-(\ell_c-\ell_b+1)};$
\item[\rm(ii)]
If $a(\xi),b(\xi)$ have a common factor $d(\xi)$, then $d(\xi)$ divides $c(\xi);$
\item[\rm(iii)]
If $a(\xi),b(\xi)$ have no common factor, then $\ell_c\ge\max(2\ell_a-1,\ell_a+\ell_b)$.
\end{enumerate}
\end{prop}
See Appendix~A for a proof of Proposition~\ref{prop:3b}.


\section{Proof of Theorem~\ref{thm:main1}}

Based on Corollary~\ref{cor:PD3},
 we only have to prove the meromorphic nonintegrability of \eqref{eqn:PD3p0}
 near the $x_3$-plane.
We set $\xi=x_3$ and $\eta=r$ and apply Theorem~\ref{thm:c} to \eqref{eqn:PD3p0}
 with the assistance of Propositions~\ref{prop:3a} and \ref{prop:3b}.
We divide the proof into the following eight cases
 and check $\deg(\kappa_{1\d})>\deg(\kappa_{1\n})$,
 condition~(H1) and the hypotheses of Proposition~\ref{prop:3a},
 using Proposition~\ref{prop:3b} if necessary:
\begin{enumerate}
\setlength{\leftskip}{-1.8em}
\item[(a)]
$\alpha_3\beta_6\neq 0$, and $\alpha_1/\alpha_3$ or $\beta_2/\beta_6\not\in\Qset$:
\begin{enumerate}
\setlength{\leftskip}{-1.8em}
\item[(a1)]
$\alpha_1\alpha_4,\beta_2/\beta_6-\alpha_1/\alpha_3,\beta_5/\beta_6-\alpha_4/\alpha_3\neq 0$;
\item[(a2)]
$\alpha_1\alpha_4,\alpha_1\beta_5-\alpha_3\beta_1-\alpha_4\beta_2\neq 0$
 and $\beta_2/\beta_6=\alpha_1/\alpha_3$;
\item[(a3)]
$\alpha_1\alpha_4\neq 0$
 and $\beta_5/\beta_6=\alpha_4/\alpha_3$;
\item[(a4)]
$\alpha_1,\beta_2/\beta_6-\alpha_1/\alpha_3\neq 0$ and $\alpha_4=0$;
\item[(a5)]
$\alpha_1,\alpha_3\beta_1-\alpha_1\beta_5\neq 0$
 and $\alpha_4,\beta_2/\beta_6-\alpha_1/\alpha_3=0$;
\item[(a6)]
$\alpha_1=0$, and $\beta_5/\beta_6-\alpha_4/\alpha_3$
 or $\alpha_3\beta_1-\alpha_4\beta_2\neq 0$;
\end{enumerate}
\item[(b)]
$\alpha_3=0$, $\beta_6\neq 0$, and $\alpha_1\neq 0$ or $\beta_2/\beta_6\not\in\Qset$:
\begin{enumerate}
\setlength{\leftskip}{-1.8em}
\item[(b1)]
$\alpha_1\alpha_4\neq 0$;
\item[(b2)]
$\alpha_1\neq 0$, $\alpha_4=0$, and $\beta_1\neq\beta_5$.
\end{enumerate}
\end{enumerate}
Equation~\eqref{eqn:fol} becomes
\begin{equation}\label{eqn:folFH}
r'=\frac{r(\alpha_1 x_3+\beta_1r^2+\beta_2x_3^2)}
 {\alpha_3x_3^2+\alpha_4r^2+(\beta_5r^2+\beta_6x_3^2)x_3},
\end{equation}
where the prime represents differentiation with respect to $x_3$.
We take $r=0$ as the integral curve, i.e., $\varphi(x_3)=0$,
 and compute \eqref{eqn:kappa} as
\begin{equation}
\begin{split}
&
\kappa_1(x_3)=\frac{\beta_2x_3+\alpha_1}{x_3(\beta_6x_3+\alpha_3)},\quad
\kappa_2(x_3)=0,\\
&
\kappa_3(x_3)=\frac{6((\beta_1\beta_6-\beta_2\beta_5)x_3^2
 -(\alpha_1\beta_5-\alpha_3\beta_1+\alpha_4\beta_2)x_3-\alpha_1\alpha_4)}
 {x_3^3(\beta_6x_3+\alpha_3)^2},\\
&
\kappa_4(x_3)=0,
\quad\ldots.
\end{split}
\label{eqn:kfH}
\end{equation}
Henceforth we take $k=3$ in application
 of Theorem~\ref{thm:c} and Proposition~\ref{prop:3a}.

\subsection{Case of $\alpha_3\beta_6\neq 0$,
 and $\alpha_1/\alpha_3$ or $\beta_2/\beta_6\not\in\Qset$}

We first consider the case of $\alpha_3\beta_6\neq 0$,
 and $\alpha_1/\alpha_3$ or $\beta_2/\beta_6\not\in\Qset$.
From \eqref{eqn:kfH} we easily see that $\deg(\kappa_{1\d})>\deg(\kappa_{1\n})$, and compute
\[
\Omega(x_3)=\exp\left(\int\frac{\beta_2x_3+\alpha_1}{x_3(\beta_6x_3+\alpha_3)}\d x_3\right)
=x_3^{\alpha_1/\alpha_3}(\beta_6 x_3+\alpha_3)^{\beta_2/\beta_6-\alpha_1/\alpha_3},
\]
so that condition~(H1) holds
 since $\alpha_1/\alpha_3$ or $\beta_2/\beta_6-\alpha_1/\alpha_3\not\in\Qset$.
We now only have to check the hypotheses of Proposition~\ref{prop:3a}.
We separately consider cases (a1)-(a6) stated above.

\subsubsection{Case of $\alpha_1\alpha_4,\beta_2/\beta_6-\alpha_1/\alpha_3,
  \beta_5/\beta_6-\alpha_4/\alpha_3\neq 0$}

By \eqref{eqn:kfH} we have
\begin{align*}
&
\kappa_{1\n}(x_3)=\frac{\beta_2x_3+\alpha_1}{\beta_6}\not\equiv 0,\quad
\kappa_{1\d}(x_3)=x_3\left(x_3+\frac{\alpha_3}{\beta_6}\right),\\
&
\kappa_{3\n}(x_3)=\frac{6}{\beta_6^2}((\beta_1\beta_6-\beta_2\beta_5)x_3^2
 -(\alpha_1\beta_5-\alpha_3\beta_1+\alpha_4\beta_2)x_3-\alpha_1\alpha_4)\not\equiv 0,\\
&
\kappa_{3\d}(x_3)=x_3^3\left(x_3+\frac{\alpha_3}{\beta_6}\right)^2,
\end{align*}
from which $n_1=2$, $\xi_{11}=0$, $\xi_{12}=-\alpha_3/\beta_6\neq 0$,
 $a_{11}=2$, $a_{12}=1$ and $n_3=0$.
Note that $\kappa_{3\n}(x_3)$ has no root at $x_3=0,-\alpha_3/\beta_6$ since
\[
\kappa_{3\n}(0)=-\frac{6\alpha_1\alpha_4}{\beta_6^2}\neq 0,\quad
\kappa_{3\n}\left(-\frac{\alpha_3}{\beta_6}\right)=\frac{6}{\beta_6^4}
 (\alpha_3\beta_2-\alpha_1\beta_6)(\alpha_3\beta_5-\alpha_1\beta_6)\neq 0.
\]
We compute \eqref{eqn:bkappa} as
\[
\bar{\kappa}_{3b}(x_3)
 =\left(\left(\frac{2\beta_2}{\beta_6}-(b_1+b_2)-1\right)x_3
 +\frac{2\alpha_1-\alpha_3(b_1+1)}{\beta_6}\right)x_3
 \left(x_3+\frac{\alpha_3}{\beta_6}\right),
\]
where $b=(b_1,b_2)\in\Nset^2$.
We see that the zero $x_3=0$  (resp. $x_3=-\alpha_3/\beta_6$) of $\bar{\kappa}_{3b}(x_3)$
 is simple if and only if
\begin{equation}
\alpha_3(b_1+1)-2\alpha_1\neq 0\quad
\mbox{(resp. $\alpha_3\beta_6b_2-2(\alpha_3\beta_2-\alpha_1\beta_6)\neq 0$)}.
\label{eqn:4a}
\end{equation}
Hence, if $\alpha_1/\alpha_3=1$
 (resp. $\beta_2/\beta_6-\alpha_1/\alpha_3=\tfrac{1}{2}$),
 then the zero $x_3=0$ (resp. $x_3=-\alpha_3/\beta_6$) of $\bar{\kappa}_3(x_3)$ is double
 and the zero $x_3=-\alpha_3/\beta_6$ (resp. $x_3=0$) of $\bar{\kappa}_{3b}(x_3)$ is simple
 for any $b\in\Nset^2$ since $\alpha_1/\alpha_3$ or $\beta_2/\beta_6\in\Qset$. 
Thus, condition~(ii) of Proposition~\ref{prop:3a} holds
 if $\alpha_1/\alpha_3=1$
  or $\beta_2/\beta_6-\alpha_1/\alpha_3=\tfrac{1}{2}$.

We assume that $2\alpha_1/\alpha_3\not\in\Zset_{\ge 2}$
 and $2(\beta_2/\beta_6-\alpha_1/\alpha_3)\not\in\Zset_{\ge 1}$.
Both conditions in \eqref{eqn:4a} hold,
 so that the zeros $x_3=0$ and $-\alpha_3/\beta_6$ of $\bar{\kappa}_{3b}(x_3)$ are simple,
 for any $b\in\Nset^2$.
Equation~\eqref{eqn:rhok} becomes
\[
\rho_3(x_3)
 =\left(\frac{2\beta_2}{\beta_6}-3\right)x_3
 +\frac{2(\alpha_1-\alpha_3)}{\beta_6}.
\]
Thus, $\deg(\kappa_{1\d})=2$, $\deg(\rho_3)\leq 1$,
 and $\kappa_{1\d}(x_3)$ and $\rho_3(x_3)$ have no common factor
 since otherwise $\rho_3(x_3)$ has a root at $x_3=0$ or $-\alpha_3/\beta_6$ and consequently
\[
\alpha_1/\alpha_3=1\quad\mbox{or}\quad
\beta_2/\beta_6-\alpha_1/\alpha_3=\tfrac{1}{2}.
\]
If $\beta_1\beta_6-\beta_2\beta_5\neq 0$,
 then $\deg(\kappa_{3\n})=2$ and
 Eq.~\eqref{eqn:prop3a} has the form \eqref{eqn:prop3b}
 with $\ell_a=3$, $\ell_b\leq 1$, $\ell_c=2$ and $b(\xi)/c(\xi)\not\in\Cset$.
On the other hand,
 if $\beta_1\beta_6-\beta_2\beta_5=0$ and Eq.~\eqref{eqn:thm1a} holds,
 then $\deg(\kappa_{3\n})\le 1$ and Eq.~\eqref{eqn:prop3a} has the form \eqref{eqn:prop3b}
 with $\ell_a=2$, $\ell_b,\ell_c\leq 1$ and $b(\xi)/c(\xi)\not\in\Cset$ since
 \begin{align*}
&
2(\alpha_1-\alpha_3)(\alpha_1\beta_5-\alpha_3\beta_1+\alpha_4\beta_2)
 -\alpha_1\alpha_4(2\beta_2-3\beta_6)\\
&
=2(\alpha_1-\alpha_3)(\alpha_1\beta_5-\alpha_3\beta_1)
 +\alpha_4(3\alpha_1\beta_6-2\alpha_3\beta_2)\neq 0
\end{align*}
by \eqref{eqn:thm1a}.
Using Propositions~\ref{prop:3a}(iii) and \ref{prop:3b}(iii),
 we obtain Theorem~\ref{thm:main1}(i).

\begin{rmk}
\label{rmk:4a}
In the above case,
 we can take $k=5$ in application of Theorem~$\ref{thm:c}$ and Proposition~$\ref{prop:3a}$
 but cannot improve our result.
Actually, we have
\begin{align*}
\kappa_{5\n}(x_3)=& -\frac{120}{\beta_6^3}(\beta_5x_3+\alpha_4)
 ((\beta_1\beta_6-\beta_2\beta_5)x_3^2\\
&\qquad
 -(\alpha_1\beta_5-\alpha_3\beta_1+\alpha_4\beta_2)x_3-\alpha_1\alpha_4)
 \not\equiv 0,\\
\kappa_{5\d}(x_3)=&x_3^5\left(x_3+\frac{\alpha_3}{\beta_5}\right)^3,
\end{align*}
from which $n_1=2$, $\xi_{11}=0$, $\xi_{12}=-\alpha_3/\beta_6\neq 0$,
 $a_{11}=4$, $a_{12}=2$ and $n_5=0$.
We compute \eqref{eqn:bkappa} as
\[
\bar{\kappa}_{5b}(x_3)
 =\left(\left(\frac{4\beta_2}{\beta_6}-(b_1+b_2)-4\right)x_3
 +\frac{4\alpha_1-\alpha_3(b_1+3)}{\beta_6}\right)x_3\left(x_3+\frac{\alpha_3}{\beta_6}\right),
\]
where $b=(b_1,b_2)\in\Nset^2$.
We see that the zero $x_3=0$ $($resp. $x_3=-\alpha_3/\beta_6)$ of $\bar{\kappa}_{3b}(x_3)$
 is simple if and only if
\[
\alpha_3(b_1+3)-4\alpha_1\neq 0\quad
(\mbox{resp. $\alpha_3\beta_6(b_2+1)-4(\alpha_3\beta_2-\alpha_1\beta_6)\neq 0$}).
\]
Hence, if $\alpha_1/\alpha_3=1$
 $($resp. $\beta_2/\beta_6-\alpha_1/\alpha_3=\tfrac{1}{2})$,
 then the zero $x_3=0$ $($resp. $x_3=-\alpha_3/\beta_6)$ of $\bar{\kappa}_3(x_3)$ is double,
 so that the zero $x_3=-\alpha_3/\beta_6$ $($resp. $x_3=0)$ of $\bar{\kappa}_{3b}(x_3)$ is simple,
 for any $b\in\Nset^2$. 
Thus, condition~{\rm(ii)} of Proposition~$\ref{prop:3a}$ holds
 if $\alpha_1/\alpha_3=1$ or $\beta_2/\beta_6-\alpha_1/\alpha_3=\tfrac{1}{2}$.
So we assume that $4\alpha_1/\alpha_3\not\in\Zset_{\ge 4}$
 and $4(\beta_2/\beta_6-\alpha_1/\alpha_3)\not\in\Zset_{\ge 2}$
 and repeat the above argument,
 but still need to exclude the case of $2\alpha_1/\alpha_3\in\Zset_{\ge 2}$
 and $2(\beta_2/\beta_6-\alpha_1/\alpha_3)\not\in\Zset_{\ge 1}$.
We do not repeat such an argument for \eqref{eqn:PD3} below.
\end{rmk}

\subsubsection{Case of $\alpha_1\alpha_4,\alpha_1\beta_5-\alpha_3\beta_1-\alpha_4\beta_2\neq 0$
 and $\beta_2/\beta_6=\alpha_1/\alpha_3$}\hspace*{-4pt}
We immediately have $\alpha_1/\alpha_3=\beta_2/\beta_6\not\in\Qset$.
By \eqref{eqn:kfH} we have
\begin{align*}
&
\kappa_{1\n}(x_3)=\frac{\beta_2}{\beta_6}\neq 0,\quad
\kappa_{1\d}(x_3)=x_3,\\
&
\kappa_{3\n}(x_3)=\frac{6}{\beta_6^2}
 ((\beta_1\beta_6-\beta_2\beta_5)x_3-\alpha_4\beta_2)\not\equiv 0,\quad
\kappa_{3\d}(x_3)=x_3^3\left(x_3+\frac{\alpha_3}{\beta_6}\right).
\end{align*}
Note that $\beta_2=\alpha_1\beta_6/\alpha_3\neq 0$ and
\[
\kappa_{3\n}\left(-\frac{\alpha_3}{\beta_6}\right)
 =\frac{6}{\beta_6^2}(\alpha_1\beta_5-\alpha_3\beta_1-\alpha_4\beta_2)\neq 0.
\]
We have $n_1=1$, $\xi_{11}=0$, $a_{11}=2$ and $n_3=0$.
 $\xi_{31}=-\alpha_3/\beta_6$ and $a_{31}=1$.
Hence, condition~(i) of Proposition~\ref{prop:3a} holds,
and we obtain Theorem~\ref{thm:main1}(ii).

\subsubsection{Case of $\alpha_1\alpha_4\neq 0$ and $\beta_5/\beta_6=\alpha_4/\alpha_3$}

By \eqref{eqn:kfH} we have
\[
\kappa_{1\n}(x_3)=\frac{\beta_2x_3+\alpha_1}{\beta_6}\not\equiv 0,\quad
\kappa_{1\d}(x_3)=x_3\left(x_3+\frac{\alpha_3}{\beta_6}\right)
\]
and
\[
\kappa_{3\n}(x_3)=\frac{6}{\beta_6^2}
 ((\beta_1\beta_6-\beta_2\beta_5)x_3-\alpha_1\beta_5)\not\equiv 0,\quad
\kappa_{3\d}(x_3)=x_3^3\left(x_3+\frac{\alpha_3}{\beta_6}\right)^2
\]
if $\alpha_1\beta_5+\alpha_3\beta_1-\alpha_4\beta_2\neq 0$; and
\[
\kappa_{3\n}(x_3)=\frac{6}{\beta_6^2}(\beta_1\beta_6-\beta_2\beta_5)\neq 0,\quad
\kappa_{3\d}(x_3)=x_3^3\left(x_3+\frac{\alpha_3}{\beta_6}\right)
\]
if $\alpha_1\beta_5+\alpha_3\beta_1-\alpha_4\beta_2=0$.
Note that $\beta_5=\alpha_4\beta_6/\alpha_3\neq 0$
 and that if $\alpha_1\beta_5+\alpha_3\beta_1-\alpha_4\beta_2\neq 0$, then
\[
\kappa_{3\n}\left(-\frac{\alpha_3}{\beta_6}\right)
 =-\frac{6}{\beta_6^2}(\alpha_1\beta_5+\alpha_3\beta_1-\alpha_4\beta_2)\neq 0.
\]
Moreover, if $\alpha_1\beta_5+\alpha_3\beta_1-\alpha_4\beta_2=0$, then
\[
\beta_1\beta_6-\beta_2\beta_5=\frac{\beta_6}{\alpha_3}(\alpha_3\beta_1-\alpha_4\beta_2)
 =-\frac{\alpha_1\beta_5\beta_6}{\alpha_3}\neq 0.
\]
We separately discuss the two cases,
 when $\alpha_1\beta_5+\alpha_3\beta_1-\alpha_4\beta_2\neq 0$ or not.

We begin with the first case.
We have $n_1=2$, $\xi_{11}=0$, $\xi_{12}=-\alpha_3/\beta_6$,
 $a_{11}=2$, $a_{12}=1$ and $n_3=0$,
 and compute \eqref{eqn:bkappa} as
\[
\bar{\kappa}_{3b}(x_3)
 =\left(\left(\frac{2\beta_2}{\beta_6}-(b_1+b_2)-1\right)x_3
 +\frac{2\alpha_1-\alpha_3(b_1+1)}{\beta_6}\right)x_3
 \left(x_3+\frac{\alpha_3}{\beta_6}\right),
\]
where $b=(b_1,b_2)\in\Nset^2$.
The zero $x_3=0$ (resp. $x_3=-\alpha_3/\beta_6$) of $\bar{\kappa}_{3b}(x_3)$
 is simple if and only if
\begin{equation}
\alpha_3(b_1+1)-2\alpha_1\neq 0\quad
(\mbox{resp. $\alpha_3\beta_6b_2-2(\alpha_3\beta_2-\alpha_1\beta_6)\neq 0$}).
\label{eqn:4c}
\end{equation}
Hence, if $\alpha_1/\alpha_3=1$
 (resp. $\beta_2/\beta_6-\alpha_1/\alpha_3=\tfrac{1}{2}$),
 then the zero $x_3=0$  (resp. $x_3=-\alpha_3/\beta_6$) of $\bar{\kappa}_3(x_3)$ is double.
Thus, condition~(ii) of Proposition~\ref{prop:3a} holds
 if $\alpha_1/\alpha_3=1$ or $\beta_2/\beta_6-\alpha_1/\alpha_3=\tfrac{1}{2}$.

Suppose that $2\alpha_1/\alpha_3\not\in\Zset_{\ge 2}$,
 $2(\beta_2/\beta_6-\alpha_1/\alpha_3)\not\in\Zset_{\ge 1}$
 and Eq.~\eqref{eqn:thm1b} holds.
Then condition \eqref{eqn:4c} holds,
 so that the zeros  $x_3=0$ and $-\alpha_3/\beta_6$ of $\bar{\kappa}_{3b}(x_3)$ are simple,
 for any $b\in\Nset^2$.
Equation~\eqref{eqn:rhok} becomes
\[
\rho_3(x_3)
 =\left(\frac{2\beta_2}{\beta_6}-3\right)x_3
 +\frac{2(\alpha_1-\alpha_3)}{\beta_6},
\]
so that $\deg(\kappa_{1\d})=2$, $\deg(\rho_3),\deg(\kappa_{3\n})\le 1$,
 and $\kappa_{1\d}(x_3)$ and $\rho_3(x_3)$ have no common factor as in Section~4.1.1.
Thus, Eq.~\eqref{eqn:prop3a} has the form \eqref{eqn:prop3b}
 with $\ell_a=2$, $\ell_b,\ell_c\le 1$ and $b(\xi)/c(\xi)\not\in\Cset$ since
\begin{align*}
&
2(\alpha_1-\alpha_3)(\beta_1\beta_6-\beta_2\beta_5)+\alpha_1\beta_5(2\beta_2-3\beta_6)\\
&
=\alpha_1\beta_6(2\beta_1-3\beta_5)-2\alpha_3(\beta_1\beta_6-\beta_2\beta_5)\neq 0
\end{align*}
by \eqref{eqn:thm1b}.
Using Propositions~\ref{prop:3a}(iii) and \ref{prop:3b}(iii),
 we obtain Theorem~\ref{thm:main1}(iiia).

We turn to the second case.
We have $n_1=1$, $\xi_{11}=0$, $a_{11}=2$ and $n_3=0$,
 and compute \eqref{eqn:bkappa} as
\[
\bar{\kappa}_{3b}(x_3)
 =\left(\left(\frac{2\beta_2}{\beta_6}-b_1-1\right)x_3
 +\frac{2\alpha_1-\alpha_3(b_1+1)}{\beta_6}\right)x_3,
\]
where $b=b_1\in\Nset$.
The zero $x_3=0$ of $\bar{\kappa}_{3b}(x_3)$ is simple if and only if
\begin{equation}
\alpha_3(b_1+1)-2\alpha_1\neq 0.
\label{eqn:4d}
\end{equation}
Hence, if $\alpha_1/\alpha_3=1$,
 then the zero $x_3=0$ of $\bar{\kappa}_3(x_3)$ is double,
 so that condition~(ii) of Proposition~\ref{prop:3a} holds.

Suppose that $2\alpha_1/\alpha_3\not\in\Zset_{\ge 2}$ and $\beta_2\neq\beta_6$.
Then condition \eqref{eqn:4d} holds,
 so that the zeros  $x_3=0,-\alpha_3/\beta_6$ of $\bar{\kappa}_{3b}(x_3)$ are simple,
 for any $b\in\Nset$.
Equation~\eqref{eqn:rhok} becomes
\[
\rho_3(x_3)
 =2\left(\frac{\beta_2}{\beta_6}-1\right)x_3
 +\frac{2(\alpha_1-\alpha_3)}{\beta_6},
\]
so that $\deg(\kappa_{1\d})=2$, $\deg(\rho_3)=1$ and $\deg(\kappa_{3n})=0$.
If $\beta_2/\beta_6=\alpha_1/\alpha_3$,
 then $x_3+\alpha_3/\beta_6$ is a common factor of $\kappa_{1\d}(x_3)$ and $\rho_3(x_3)$ by
\[
\rho_3\left(-\frac{\alpha_3}{\beta_6}\right)=\frac{2}{\beta_6^2}(\alpha_1\beta_6-\alpha_3\beta_2),
\] 
and Eq.~\eqref{eqn:prop3a} has the form \eqref{eqn:prop3b}
 such that $a(\xi)$ and $b(\xi)$ have a common factor and $b(\xi)/c(\xi)\not\in\Cset$ but $\ell_c=0$.
On the other hand, if $\beta_2/\beta_6=\alpha_1/\alpha_3$,
 then $\kappa_{1\d}(x_3)$ and $\rho_3(x_3)$ have no common factor
 and Eq.~\eqref{eqn:prop3a} has the form \eqref{eqn:prop3b}
 with $\ell_a=2$, $\ell_b=1$, $\ell_c=0$ and $b(\xi)/c(\xi)\not\in\Cset$.
Using Proposition~\ref{prop:3a}(ii) and (iii),
 we see that condition~(iii) of Proposition~\ref{prop:3a} holds,
 and obtain Theorem~\ref{thm:main1}(iiib).

\subsubsection{Case of $\alpha_1,\beta_2/\beta_6-\alpha_1/\alpha_3\neq 0$ and $\alpha_4=0$}

By \eqref{eqn:kfH} we have
\[
\kappa_{1\n}(x_3)=\frac{\beta_2x_3+\alpha_1}{\beta_6}\not\equiv 0,\quad
\kappa_{1\d}(x_3)=x_3\left(x_3+\frac{\alpha_3}{\beta_6}\right),
\]
and
\begin{align*}
&
\kappa_{3\n}(x_3)=\frac{6}{\beta_6^2}((\beta_1\beta_6-\beta_2\beta_5)x_3
 +\alpha_3\beta_1-\alpha_1\beta_5)\not\equiv 0,\\
&
\kappa_{3\d}(x_3)=x_3^2\left(x_3+\frac{\alpha_3}{\beta_6}\right)^2
\end{align*}
if $\beta_5,\alpha_3\beta_1-\alpha_1\beta_5\neq 0$;
\[
\kappa_{3\n}(x_3)=\frac{6}{\beta_6^2}(\beta_1\beta_6-\beta_2\beta_5)\neq 0,\quad
\kappa_{3\d}(x_3)=x_3\left(x_3+\frac{\alpha_3}{\beta_6}\right)^2
\]
if $\alpha_3\beta_1-\alpha_1\beta_5=0$ and $\beta_5\neq 0$; and
\[
\kappa_{3\n}(x_3)=\frac{6\beta_1}{\beta_6}\neq 0,\quad
\kappa_{3\d}(x_3)=x_3^2\left(x_3+\frac{\alpha_3}{\beta_6}\right)
\]
if $\beta_5=0$ and $\beta_1\neq 0$.
Note that if $\alpha_3\beta_1-\alpha_1\beta_5=0$ and $\beta_5\neq 0$, then
\[
\beta_1\beta_6-\beta_2\beta_5=\frac{\beta_5}{\alpha_3}(\alpha_1\beta_6-\alpha_3\beta_2)\neq 0.
\]
We separately discuss the three cases,
 when $\beta_5,\alpha_3\beta_1-\alpha_1\beta_5\neq 0$;
 when $\alpha_3\beta_1-\alpha_1\beta_5=0$ and $\beta_5\neq 0$;
 and when $\beta_5=0$ and $\beta_1\neq 0$.

We begin with the first case.
We have
 $n_1=2$, $\xi_{11}=0$, $\xi_{12}=-\alpha_3/\beta_6$,
 $a_{11}=a_{12}=1$ and $n_3=0$.
We compute \eqref{eqn:bkappa} as
\[
\bar{\kappa}_{3b}(x_3)
 =\left(\left(\frac{2\beta_2}{\beta_6}-(b_1+b_2)\right)x_3
 +\frac{2\alpha_1-\alpha_3b_1}{\beta_6}\right)x_3\left(x_3+\frac{\alpha_3}{\beta_6}\right),
\]
where $b=(b_1,b_2)\in\Nset^2$.
The zero $x_3=0$ (resp. $x_3=-\alpha_3/\beta_6$) of $\bar{\kappa}_{3b}(x_3)$ is simple
 if and only if
\begin{equation}
\alpha_3b_1-2\alpha_1\neq 0\quad
\mbox{(resp. $\alpha_3\beta_6b_2-2(\alpha_3\beta_2-\alpha_1\beta_6)\neq 0$)}
\label{eqn:4e}
\end{equation}
Hence, if $\alpha_1/\alpha_3=\tfrac{1}{2}$
 (resp. $\beta_2/\beta_6-\alpha_1/\alpha_3=\tfrac{1}{2}$),
 then the zero $x_3=0$ (resp. $x_3=-\alpha_3/\beta_6$) of $\bar{\kappa}_3(x_3)$ is double
 and the zero $x_3=-\alpha_3/\beta_6$ (resp. $x_3=0$) of $\bar{\kappa}_{3b}(x_3)$ is simple
 for any $b\in\Nset^2$. 
Thus, condition~(ii) of Proposition~\ref{prop:3a} holds
 if $\alpha_1/\alpha_3$ or $\beta_2/\beta_6-\alpha_1/\alpha_3=\tfrac{1}{2}$.

Suppose that
 $2\alpha_1/\alpha_3,2(\beta_2/\beta_6-\alpha_1/\alpha_3)\not\in\Zset_{\ge 1}$
 and Eq.~\eqref{eqn:thm1c} holds.
Then both conditions in \eqref{eqn:4e} hold,
 so that the zeros  $x_3=0,-\alpha_3/\beta_6$ of $\bar{\kappa}_{3b}(x_3)$ are simple,
 for any $b\in\Nset^2$.
Equation~\eqref{eqn:rhok} becomes
\[
\rho_3(x_3)
 =2\left(\frac{\beta_2}{\beta_6}-1\right)x_3
 +\frac{2\alpha_1-\alpha_3}{\beta_6},
\]
so that $\deg(\kappa_{1\d})=2$, $\deg(\rho_3),\deg(\kappa_{3\n})\le 1$,
 and $\kappa_{1\d}(x_3)$ and $\rho_3(x_3)$ have no common factor
 as in Section~4.1.1.
Thus, Eq.~\eqref{eqn:prop3a} has the form \eqref{eqn:prop3b}
 with $\ell_a=2$, $\ell_b,\ell_c\le 1$ and $b(\xi)/c(\xi)\not\in\Cset$
 since
\begin{align*}
&
(2\alpha_1-\alpha_3)(\beta_1\beta_6-\beta_2\beta_5)
 -2(\beta_2-\beta_6)(\alpha_3\beta_1-\alpha_1\beta_5)\\
&
=2\alpha_1\beta_6(\beta_1-\beta_5)
 +\alpha_3(\beta_1\beta_6+\beta_2\beta_5-2\beta_1\beta_2)\neq 0
\end{align*}
by \eqref{eqn:thm1c}.
Using Propositions~\ref{prop:3a}(iii) and \ref{prop:3b}(iii),
 we obtain Theorem~\ref{thm:main1}(iva).

We turn to the second case.
We have $n_1=1$, $\xi_{11}=-\alpha_3/\beta_6\neq 0$, $a_{11}=1$ and $n_3=0$.
We compute \eqref{eqn:bkappa} as
\[
\bar{\kappa}_{3b}(x_3)
 =\left(\left(\frac{2\beta_2}{\beta_6}-b_1\right)x_3
 +\frac{2\alpha_1}{\beta_6}\right)\left(x_3+\frac{\alpha_3}{\beta_6}\right),
\]
where $b=b_1\in\Nset$.
The zero $x_3=-\alpha_3/\beta_6$ of $\bar{\kappa}_{3b}(x_3)$ is simple if and only if
\begin{equation}
\alpha_3\beta_6b_1-2(\alpha_3\beta_2-\alpha_1\beta_6)\neq 0.
\label{eqn:4f}
\end{equation}
Hence, if $\beta_2/\beta_6-\alpha_1/\alpha_3=\tfrac{1}{2}$,
 then the zero $x_3=-\alpha_3/\beta_6$ of $\bar{\kappa}_3(x_3)$ is double.
Thus, condition~(ii) of Proposition~\ref{prop:3a} holds
 and $\beta_2/\beta_6\neq\tfrac{1}{2}$
 if $\beta_2/\beta_6-\alpha_1/\alpha_3=\tfrac{1}{2}$.
 
Suppose that $2(\beta_2/\beta_6-\alpha_1/\alpha_3)\not\in\Zset_{\ge 1}$
 and $\beta_2/\beta_6\neq\tfrac{1}{2}$.
Then condition \eqref{eqn:4f} holds,
 so that the zero $x_3=-\alpha_3/\beta_6$ of $\bar{\kappa}_{3b}(x_3)$ is simple,
 for any $b\in\Nset$.
Equation~\eqref{eqn:rhok} becomes
\[
\rho_3(x_3)
 =\left(\frac{2\beta_2}{\beta_6}-1\right)x_3+\frac{2\alpha_1}{\beta_6},
\]
so that $\deg(\kappa_{1\d})=2$, $\deg(\rho_3)=1$,  $\deg(\kappa_{3\n})=0$,
 and $\kappa_{1\d}(x_3)$ and $\rho_3(x_3)$ have no common factor as in Section~4.1.1.
Thus, Eq.~\eqref{eqn:prop3a} has the form \eqref{eqn:prop3b}
 with $\ell_a=2$, $\ell_b=1$, $\ell_c=0$ and $b(\xi)/c(\xi)\not\in\Cset$.
Using Propositions~\ref{prop:3a}(iii) and \ref{prop:3b}(iii),
 we obtain Theorem~\ref{thm:main1}(ivb).
 
We turn to the third case.
We have $n_1=1$, $\xi_{11}=0$, $a_{11}=1$ and $n_3=0$.
We compute \eqref{eqn:bkappa} as
\[
\bar{\kappa}_{3b}(x_3)
 =\left(\left(\frac{2\beta_2}{\beta_6}-b_1\right)x_3
 +\frac{2\alpha_1-\alpha_3b_1}{\beta_6}\right)x_3,
\]
where $b=b_1\in\Nset$.
The zero $x_3=0$ of $\bar{\kappa}_{3b}(x_3)$ is simple if and only if
\begin{equation}
\alpha_3b_1-2\alpha_1\neq 0.
\label{eqn:4g}
\end{equation}
Hence, if $\alpha_1/\alpha_3=\tfrac{1}{2}$,
 then the zero $x_3=0$ of $\bar{\kappa}_3(x_3)$ is double.
Thus, condition~(ii) of Proposition~\ref{prop:3a} holds
 and $\beta_2/\beta_6\not\in\Qset$
 if $\alpha_1/\alpha_3=\tfrac{1}{2}$.

Suppose that $2\alpha_1/\alpha_3\not\in\Zset_{\ge 1}$
 and $\beta_2/\beta_6\neq\tfrac{1}{2}$.
Then condition \eqref{eqn:4g} holds
 so that the zero $x_3=0$ of $\bar{\kappa}_{3b}(x_3)$ is simple for any $b\in\Nset$.
Equation~\eqref{eqn:rhok} becomes
\[
\rho_3(x_3)
 =\left(\frac{2\beta_2}{\beta_6}-1\right)x_3+\frac{2\alpha_1-\alpha_3}{\beta_6},
\]
so that $\deg(\kappa_{1\d})=2$, $\deg(\rho_3)=1$, $\deg(\kappa_{3\n})=0$,
 and $\kappa_{1\d}(x_3)$ and $\rho_3(x_3)$ have no common factor as in Section~4.1.1.
Thus, Eq.~\eqref{eqn:prop3a} has the form \eqref{eqn:prop3b}
 with $\ell_a=2$, $\ell_b=1$, $\ell_c=0$ and $b(\xi)/c(\xi)\not\in\Cset$.
Using Propositions~\ref{prop:3a}(iii) and \ref{prop:3b}(iii),
 we obtain Theorem~\ref{thm:main1}(ivc).
 
\subsubsection{Case of $\alpha_1,\alpha_3\beta_1-\alpha_1\beta_5\neq 0$ and
 $\alpha_4,\beta_2/\beta_6-\alpha_1/\alpha_3=0$}

We immediately have $\alpha_1/\alpha_3=\beta_2/\beta_6\not\in\Qset$.
By \eqref{eqn:kfH} we have
\begin{align*}
&
\kappa_{1\n}(x_3)=\frac{\beta_2}{\beta_6}\neq 0,\quad
\kappa_{1\d}(x_3)=x_3,\\
&
\kappa_{3\n}(x_3)=\frac{6}{\beta_6^2}(\beta_1\beta_6-\beta_2\beta_5)\neq 0,\quad
\kappa_{3\d}(x_3)=x_3^2\left(x_3+\frac{\alpha_3}{\beta_6}\right).
\end{align*}
Note that
\[
\beta_1\beta_6-\beta_2\beta_5=\frac{\beta_6}{\alpha_3}(\alpha_3\beta_1-\alpha_1\beta_5)\neq 0
\]
since $\beta_2=\alpha_1\beta_6/\alpha_3$.
We have $n_1=1$, $\xi_{11}=0$, $a_{11}=1$, $n_3=1$,
 $\xi_{31}=-\alpha_3/\beta_6$ and $a_{31}=1$.
Hence, condition~(i) of Proposition~\ref{prop:3a} holds
 and we obtain Theorem~\ref{thm:main1}(v).

\subsubsection{Case of $\alpha_1=0$,
 and $\beta_5/\beta_6-\alpha_4/\alpha_3$ or $\alpha_3\beta_1-\alpha_4\beta_2\neq 0$}
 
We immediately have $\beta_2/\beta_6\not\in\Qset$.
By \eqref{eqn:kfH} we have
\[
\kappa_{1\n}(x_3)=\frac{\beta_2}{\beta_6}\neq 0,\quad
\kappa_{1\d}(x_3)=x_3+\frac{\alpha_3}{\beta_6}
\]
and
\begin{align*}
&
\kappa_{3\n}(x_3)=\frac{6}{\beta_6^2}((\beta_1\beta_6-\beta_2\beta_5)x_3
 +(\alpha_3\beta_1-\alpha_4\beta_2))\not\equiv 0,\\
&
\kappa_{3\d}(x_3)=x_3^2\left(x_3+\frac{\alpha_3}{\beta_6}\right)^2
\end{align*}
if $\beta_5/\beta_6-\alpha_4/\alpha_3,\alpha_3\beta_1-\alpha_4\beta_2\neq 0$;
\[
\kappa_{3\n}(x_3)=\frac{6}{\beta_6^2}(\beta_1\beta_6-\beta_2\beta_5)\neq 0,\quad
\kappa_{3\d}(x_3)=x_3^2\left(x_3+\frac{\alpha_3}{\beta_6}\right)
\]
if $\beta_5/\beta_6-\alpha_4/\alpha_3=0$ and $\alpha_3\beta_1-\alpha_4\beta_2\neq 0$; and
\[
\kappa_{3\n}(x_3)=\frac{6}{\beta_6^2}(\beta_1\beta_6-\beta_2\beta_5)\neq 0,\quad
\kappa_{3\d}(x_3)=x_3\left(x_3+\frac{\alpha_3}{\beta_6}\right)^2
\]
if $\alpha_3\beta_1-\alpha_4\beta_2=0$ and $\beta_5/\beta_6-\alpha_4/\alpha_3\neq 0$.
Note that
 if one of $\beta_5/\beta_6-\alpha_4/\alpha_3$
 and $\alpha_3\beta_1-\alpha_4\beta_2$ is zero and the other is not, then
\[
\beta_1\beta_6-\beta_2\beta_5
=\frac{1}{\alpha_3}(\beta_6(\alpha_3\beta_1-\alpha_4\beta_2)
 -\beta_2(\alpha_3\beta_5-\alpha_4\beta_6))\neq 0.
\]
We separately discuss the three cases,
 when $\beta_5/\beta_6-\alpha_4/\alpha_3,\alpha_3\beta_1-\alpha_4\beta_2\neq 0$;
 when $\beta_5/\beta_6-\alpha_4/\alpha_3=0$ and $\alpha_3\beta_1-\alpha_4\beta_2\neq 0$;
 and when $\alpha_3\beta_1-\alpha_4\beta_2=0$ and $\beta_5/\beta_6-\alpha_4/\alpha_3\neq 0$.
For the third case, we have
 $n_1=1$, $\xi_{11}=-\alpha_3/\beta_6$, $a_{11}=1$,
 $n_3=1$, $\xi_{31}=0$ and $a_{31}=1$,
 so that condition~(i) of Proposition~\ref{prop:3a} holds.
This yields Theorem~\ref{thm:main1}(vic).

We now discuss the first case
 and additionally assume that Eq.~\eqref{eqn:thm1d} holds.
We have $n_1=1$, $\xi_{11}=-\alpha_3/\beta_6$, $a_{11}=1$,
 $n_3=1$, $\xi_{31}=0$ and $a_{31}=2$.
We compute \eqref{eqn:bkappa} as
\[
\bar{\kappa}_{3b}(x_3)
 =\left(\frac{2\beta_2}{\beta_6}-b_1\right)\left(x_3+\frac{\alpha_3}{\beta_6}\right),
\]
where $b=b_1\in\Nset$.
The zero $x_3=-\alpha_3/\beta_6$ of $\bar{\kappa}_{3b}(x_3)$ is simple for any $b\in\Nset$.
Equation~\eqref{eqn:rhok} becomes
\[
\rho_3(x_3)
 =2\left(\frac{\beta_2}{\beta_6}-1\right)x_3-\frac{\alpha_3}{\beta_6},
\]
so that $\deg(\kappa_{1\d})=1$, $\deg(\rho_3)=1$, $\deg(\kappa_{3\n})\le 1$,
 and $\kappa_{1\d}(x_3)$ and $\rho_3(x_3)$ have no common factor
 as in Section~4.1.1.
Thus, Eq.~\eqref{eqn:prop3a} has the form \eqref{eqn:prop3b}
 with $\ell_a=2$, $\ell_b=1$, $\ell_c\leq 1$ and $b(\xi)/c(\xi)\not\in\Cset$ since
\begin{align*}
&
\alpha_3(\beta_1\beta_6-\beta_2\beta_5)+2(\beta_2-\beta_6)
 (\alpha_3\beta_1-\alpha_4\beta_2)\\
&
=\alpha_3(2\beta_1\beta_2-\beta_2\beta_5-\beta_1\beta_6)
 -2\alpha_4\beta_2(\beta_2-\beta_6)\neq 0
\end{align*}
by \eqref{eqn:thm1d}.
Using Propositions~\ref{prop:3a}(iii) and \ref{prop:3b}(iii),
 we obtain Theorem~\ref{thm:main1}(via).

We turn to the second case.
We have $n_1=0$, $n_3=1$, $\xi_{31}=0$ and $a_{31}=2$,
 and compute \eqref{eqn:rhok} as
\[
\rho_3(x_3)
=\left(\frac{2\beta_2}{\beta_6}-1\right)x_3-\frac{\alpha_3}{\beta_6},
\]
so that $\deg(\kappa_{1\d})=1$, $\deg(\rho_3)=1$, $\deg(\kappa_{3\n})=0$,
 and $\kappa_{1\d}(x_3)$ and $\rho_3(x_3)$ have no common factor as in Section~4.1.1.
Thus, Eq.~\eqref{eqn:prop3a} has the form \eqref{eqn:prop3b}
 with $\ell_a=2$, $\ell_b=1$, $\ell_c=0$ and $b(\xi)/c(\xi)\not\in\Cset$.
Using Propositions~\ref{prop:3a}(iii) and \ref{prop:3b}(iii),
  we obtain Theorem~\ref{thm:main1}(vib).

\subsection{Case of $\alpha_3=0$, $\beta_6\neq 0$,
 and $\alpha_1\neq 0$ or $\beta_2/\beta_6\not\in\Qset$}
We next consider the case of $\alpha_3=0$, $\beta_6\neq 0$
 and $\alpha_1\neq 0$ or $\beta_2/\beta_6\not\in\Qset$.
From \eqref{eqn:kfH} we easily see that $\deg(\kappa_{1\d})>\deg(\kappa_{1\n})$, and compute
\[
\Omega(x_3)=\exp\left(\int\frac{\beta_2x_3+\alpha_1}{\beta_6x_3^2}\d x_3\right)
=\exp\left(-\frac{\alpha_1}{\beta_6 x_3}\right)x_3^{\beta_2/\beta_6},
\]
so that condition~(H1) holds.
We separately consider cases~(b1) and (b2) stated at the beginning of this section
 and check the hypotheses of Proposition~\ref{prop:3a}.

\subsubsection{Case of $\alpha_1\alpha_4\neq 0$}
By \eqref{eqn:kfH} we have
\begin{align*}
&
\kappa_{1\n}(x_3)=\frac{\beta_2x_3+\alpha_1}{\beta_6}\not\equiv 0,\quad
\kappa_{1\d}(x_3)=x_3^2,\\
&
\kappa_{3\n}(x_3)=\frac{6}{\beta_6^2}((\beta_1\beta_6-\beta_2\beta_5)x_3^2
 -(\alpha_1\beta_5+\alpha_4\beta_2)x_3-\alpha_1\alpha_4)\not\equiv 0,\quad
\kappa_{3\d}(x_3)=x_3^5,
\end{align*}
from which $n_1=1$, $\xi_{11}=0$, $a_{11}=3$ and $n_3=0$.
We compute \eqref{eqn:bkappa} as
\[
\bar{\kappa}_{3b}(x_3)
 =\left(\left(\frac{2\beta_2}{\beta_6}-b_1-2\right)x_3
 +\frac{2\alpha_1}{\beta_6}\right)x_3,
\]
where $b=b_1\in\Nset$.
The zero $x_3=0$ of $\bar{\kappa}_{3b}(x_3)$ is simple for any $b\in\Nset$.
Equation~\eqref{eqn:rhok} becomes
\[
\rho_3(x_3)
 =\left(\frac{2\beta_2}{\beta_6}-3\right)x_3+\frac{2\alpha_1}{\beta_6},
\]
so that $\deg(\kappa_{1\d})=2$, $\deg(\rho_3)\le 1$,
 and $\kappa_{1\d}(x_3)$ and $\rho_3(x_3)$ have no common factor
 since otherwise $\rho_3(x_3)$ has a root at $x_3=0$
 and consequently $\alpha_1=0$.
If $\beta_1\beta_6-\beta_2\beta_5\neq 0$,
 then $\deg(\kappa_{3\n})=2$ and
 Eq.~\eqref{eqn:prop3a} has the form \eqref{eqn:prop3b} with $\ell_a=2$, 
$\ell_b\le 1$, $\ell_c=2$ and $b(\xi)/c(\xi)\not\in\Cset$.
On the other hand, if $\beta_1\beta_6-\beta_2\beta_5=0$
 and $2\alpha_1\beta_5+3\alpha_4\beta_6\neq 0$,
 then $\deg(\kappa_{3\n})\le 1$ and Eq.~\eqref{eqn:prop3a} has the form \eqref{eqn:prop3b}
 with $\ell_a=2$, $\ell_b,\ell_c\le 1$ and $b(\xi)/c(\xi)\not\in\Cset$ since
\[
2\alpha_1(\beta_1\beta_6-\beta_2\beta_5)-\alpha_1\alpha_4(2\beta_2-\beta_6)
 =\alpha_1(2\alpha_1\beta_5+3\alpha_4\beta_6)\neq 0.
\]
Using Propositions~\ref{prop:3a}(iii) and \ref{prop:3b}(iii),
 we obtain Theorem~\ref{thm:main1}(vii).

\subsubsection{Case of $\alpha_1\neq 0$, $\alpha_4=0$ and $\beta_1\neq \beta_5$}
By \eqref{eqn:kfH} we have
\[
\kappa_{1\n}(x_3)=\frac{\beta_2x_3+\alpha_1}{\beta_6}\not\equiv 0,\quad
\kappa_{1\d}(x_3)=x_3^2
\]
and
\[
\kappa_{3\n}(x_3)=\frac{6}{\beta_6^2}((\beta_1\beta_6-\beta_2\beta_5)x_3
 -\alpha_1\beta_5)\not\equiv 0,\quad
\kappa_{3\d}(x_3)=x_3^4
\]
if $\beta_5\neq 0$; and
\[
\kappa_{3\n}(x_3)=\frac{6\beta_1}{\beta_6}\neq 0,\quad
\kappa_{3\d}(x_3)=x_3^3
\]
if $\beta_5=0$ and $\beta_1\neq 0$.
We separately discuss the two cases,
 when $\beta_5\neq 0$; and when $\beta_5=0$ and $\beta_1\neq 0$.
 
We begin with the first case and additionally assume that $\beta_1\neq \beta_5$.
We have $n_1=1$, $\xi_{11}=0$, $a_{11}=2$ and $n_3=0$, and
 compute \eqref{eqn:bkappa} as
\[
\bar{\kappa}_{3b}(x_3)
 =\left(\left(\frac{2\beta_2}{\beta_6}-b_1-1\right)x_3
 +\frac{2\alpha_1}{\beta_6}\right)x_3,
\]
where $b=b_1\in\Nset$.
The zero $x_3=0$ of $\bar{\kappa}_{3b}(x_3)$ is simple for any $b\in\Nset$.
Equation~\eqref{eqn:rhok} becomes
\[
\rho_3(x_3)
 =2\left(\frac{\beta_2}{\beta_6}-1\right)x_3+\frac{2\alpha_1}{\beta_6},
\]
so that $\deg(\kappa_{1\d})=2$, $\deg(\rho_3),\deg(\kappa_{3\n})\le 1$,
 and $\kappa_{1\d}(x_3)$ and $\rho_3(x_3)$ have no common factor  as in Section~4.2.1.
Thus, Eq.~\eqref{eqn:prop3a} has the form \eqref{eqn:prop3b}
  with $\ell_a=2$, $\ell_b,\ell_c\le 1$ and $b(\xi)/c(\xi)\not\in\Cset$ since
\[
\alpha_1(\beta_1\beta_6-\beta_2\beta_5)+\alpha_1\beta_5(\beta_2-\beta_6)
=\alpha_1\beta_6(\beta_1-\beta_5)\neq 0.
\]
Using Propositions~\ref{prop:3a}(iii) and \ref{prop:3b}(iii),
 we obtain Theorem~\ref{thm:main1}(viiia).
  
We turn to the second case
 and additionally assume that $\beta_2/\beta_6\neq\tfrac{1}{2}$.
We have $n_1=1$, $\xi_{11}=0$, $a_{11}=1$ and $n_3=0$, and compute \eqref{eqn:bkappa} as
\[
\bar{\kappa}_{3b}(x_3)
 =\left(\left(\frac{2\beta_2}{\beta_6}-b_1\right)x_3
 +\frac{2\alpha_1}{\beta_6}\right)x_3,
\]
where $b=b_1\in\Nset$.
The zero $x_3=0$ of $\bar{\kappa}_{3b}(x_3)$ is simple for any $b\in\Nset$
 and Eq.~\eqref{eqn:rhok} becomes
\[
\rho_3(x_3)
 =\left(\frac{2\beta_2}{\beta_6}-1\right)x_3+\frac{2\alpha_1}{\beta_6}.
\]
Thus, $\deg(\kappa_{1\d})=2$, $\deg(\rho_3)=1$, $\deg(\kappa_{3\n})=0$,
 and $\kappa_{1\d}(x_3)$ and $\rho_3(x_3)$ have no common factor as in Section~4.2.1.
Thus, Eq.~\eqref{eqn:prop3a} has the form \eqref{eqn:prop3b}
 with $\ell_a=2$, $\ell_b=1$, $\ell_c=0$ and $b(\xi)/c(\xi)\not\in\Cset$.
Using Propositions~\ref{prop:3a}(iii) and \ref{prop:3b}(iii),
 we obtain Theorem~\ref{thm:main1}(viiib)
 and complete the proof of Theorem~\ref{thm:main1}.
\qed


\section{Proof of Theorem~\ref{thm:main2}}

Based on Corollary~\ref{cor:PD4},
 we only have to prove the meromorphic nonintegrability of \eqref{eqn:PD4p0}
 near the $r_1$-axis as in the proof of Theorem~\ref{thm:main1}.
We set $\xi=r_2$ and $\eta=r_1$ and apply Theorem~\ref{thm:c} to \eqref{eqn:PD4p0}
 with assistance of Propositions~\ref{prop:3a} and \ref{prop:3b}.
We divide the proof into the following eight cases
 and check $\deg(\kappa_{1\d})>\deg(\kappa_{1\n})$,
 condition~(H1) and the hypotheses of Proposition~\ref{prop:3a},
 using Proposition~\ref{prop:3b} if necessary:
\begin{enumerate}
\setlength{\leftskip}{-1.8em}
\item[(a)]
$\alpha_4\beta_5\neq 0$,  and $\alpha_2/\alpha_4$ or $\beta_2/\beta_5\not\in\Qset$:
\begin{enumerate}
\setlength{\leftskip}{-1.8em}
\item[(a1)]
$\alpha_2,\alpha_1\alpha_4-\alpha_2\alpha_3,
 \beta_2/\beta_5-\alpha_2/\alpha_4,\beta_6/\beta_5-\alpha_3/\alpha_4\neq 0$;
\item[(a2)]
$\alpha_2,\alpha_1\alpha_4-\alpha_2\alpha_3,
 \alpha_1\beta_5+\alpha_2\beta_6-\alpha_3\beta_2-\alpha_4\beta_3\neq 0$
 and $\beta_2/\beta_5=\alpha_2/\alpha_4$;
\item[(a3)]
$\alpha_2,\alpha_1\alpha_4-\alpha_2\alpha_3,\beta_2/\beta_5-\alpha_2/\alpha_4\neq 0$
 and $\beta_6/\beta_5=\alpha_3/\alpha_4$;
\item[(a4)]
$\alpha_2,\beta_2/\beta_5-\alpha_2/\alpha_4,\beta_6/\beta_5-\alpha_3/\alpha_4\neq 0$
 and $\alpha_1\alpha_4=\alpha_2\alpha_3$;
\item[(a5)]
$\alpha_1\beta_2,\beta_6/\beta_5-\alpha_3/\alpha_4\neq 0$ and $\alpha_2=0$;
\item[(a6)]
$\beta_2,\beta_3/\beta_2-\alpha_3/\alpha_4\neq 0$ and $\alpha_1,\alpha_2=0$;
\end{enumerate}
\item[(b)]
$\alpha_4=0$, $\beta_5\neq 0$, and $\alpha_2\neq 0$ or $\beta_2/\beta_5\not\in\Qset$:
\begin{enumerate}
\setlength{\leftskip}{-1.8em}
\item[(b1)]
$\alpha_2\alpha_3\neq 0$;
\item[(b2)]
$\alpha_1\alpha_2\neq 0$ and $\alpha_3=0$.
\end{enumerate}
\end{enumerate}
Equation~\eqref{eqn:fol} becomes
\begin{equation}\label{eqn:folFH}
r'_1=\frac{r_1(\alpha_1 r_1^2+\alpha_2r_2^2+\beta_1r_1^4+\beta_2r_2^4+\beta_3r_1^2r_2^2)}
 {r_2(\alpha_3 r_1^2+\alpha_4r_2^2+\beta_4r_1^4+\beta_5r_2^4+\beta_6r_1^2r_2^2)},
\end{equation}
where the prime represents differentiation with respect to $r_2$.
We take $r_1=0$ as the integral curve, i.e., $\varphi(r_2)=0$,
 and compute \eqref{eqn:kappa} as
\begin{equation}
\begin{split}
\kappa_1(r_2)=&\frac{\beta_2r_2^2+\alpha_2}{r_2(\beta_5r_2^2+\alpha_4)},\quad
\kappa_2(r_2)=0,\\
\kappa_3(r_2)=&\frac{6((\beta_3\beta_5-\beta_2\beta_6)r_2^4
 +(\alpha_1\beta_5-\alpha_2\beta_6-\alpha_3\beta_2+\alpha_4\beta_3)r_2^2
 +\alpha_1\alpha_4-\alpha_2\alpha_3)}
 {r_2^3(\beta_5r_2^2+\alpha_4)^2},\\
\kappa_4(r_2)=&0,\quad
\cdots.
\end{split}
\label{eqn:kdH}
\end{equation}
Henceforth we take $k=3$
 in application of Theorem~\ref{thm:c} and Proposition~\ref{prop:3a},
 as in Section~4.

\subsection{Case of $\alpha_4\beta_5\neq 0$,
 and $\alpha_2/\alpha_4$ or $\beta_2/\beta_5\not\in\Qset$}

We first consider the case of $\alpha_4\beta_5\neq 0$
 and assume that $\alpha_2/\alpha_4$ or $\beta_2/\beta_5\not\in\Qset$.
From \eqref{eqn:kdH} we see that  $\deg(\kappa_{1\d})>\deg(\kappa_{1\n})$, and compute
\[
\Omega(r_2)=\exp\left(\int\frac{\beta_2r_2^2+\alpha_2}{r_2(\beta_5r_2^2+\alpha_4)}\d r_2\right)
=r_2^{\alpha_2/\alpha_4}(\beta_5r_2^2+\alpha_4)^{\beta_2/2\beta_5-\alpha_2/2\alpha_4},
\]
so that condition~(H1) holds
 since $\alpha_2/\alpha_4$ or $\beta_2/\beta_5-\alpha_2/\alpha_4\not\in\Qset$.
We now only have to check the hypotheses of Proposition~\ref{prop:3a}.
We separately consider cases~(a1)-(a6) stated above. 

\subsubsection{Case of $\alpha_2,\alpha_1\alpha_4-\alpha_2\alpha_3,
 \beta_2/\beta_5-\alpha_2/\alpha_4,\beta_6/\beta_5-\alpha_3/\alpha_4\neq 0$}

By \eqref{eqn:kdH} we have
\begin{align*}
\kappa_{1\n}(r_2)=&\frac{\beta_2r_2^2+\alpha_2}{\beta_5}\not\equiv 0,\quad
\kappa_{1\d}(r_2)=r_2\left(r_2^2+\frac{\alpha_4}{\beta_5}\right),\\
\kappa_{3\n}(r_2)=&\frac{6}{\beta_5^2}((\beta_3\beta_5-\beta_2\beta_6)r_2^4\\
& +(\alpha_1\beta_5-\alpha_2\beta_6-\alpha_3\beta_2+\alpha_4\beta_3)r_2^2
 +\alpha_1\alpha_4-\alpha_2\alpha_3)\not\equiv 0,\\
\kappa_{3\d}(r_2)=&r_2^3\left(r_2^2+\frac{\alpha_4}{\beta_5}\right)^2.
\end{align*}
Note that $\kappa_{3\n}(r_2)$ has no zero at $r_2=0,\pm\sqrt{-\alpha_4/\beta_5}$
 since $\alpha_1\alpha_4-\alpha_2\alpha_3\neq 0$ and
\[
\kappa_{3\n}\left(\pm\sqrt{\frac{\alpha_4}{\beta_5}}\right)
=\frac{1}{\beta_5^4}(\alpha_4\beta_2-\alpha_2\beta_5)(\alpha_3\beta_5-\alpha_4\beta_6)\neq 0.
\]
We have $n_1=3$, $\xi_{11}=0$, $\xi_{12}=\sqrt{-\alpha_4/\beta_5}$,
 $\xi_{13}=-\sqrt{-\alpha_4/\beta_5}$, $a_{11}=2$, $a_{12},a_{13}=1$ and $n_3=0$.
We compute \eqref{eqn:bkappa} as
\begin{align*}
\bar{\kappa}_{3b}(r_2)
 =&\biggl(\left(\frac{2\beta_2}{\beta_5}-(b_1+b_2+b_3)-1\right)r_2^2\\
&  -(b_2-b_3)\sqrt{-\frac{\alpha_4}{\beta_5}}r_2
 +\frac{2\alpha_2-\alpha_4(b_1+1)}{\beta_5}\biggr)
 r_2\left(r_2^2+\frac{\alpha_4}{\beta_5}\right),
\end{align*}
where $b=(b_1,b_2,b_3)\in\Nset^3$.
We see that the zero $r_2=0$
 (resp. the zeros $r_2=\sqrt{-\alpha_4/\beta_5}$ and $-\sqrt{-\alpha_4/\beta_5}$)
 of $\bar{\kappa}_{3b}(r_2)$
 is (resp. are) simple if and only if
\begin{align}
\alpha_4(b_1+1)-2\alpha_2\neq 0\quad
(&\mbox{resp. $\alpha_4\beta_5 b_2-(\alpha_4\beta_2-\alpha_2\beta_5)$}\notag\\
& \mbox{and $\alpha_4\beta_5 b_3-(\alpha_4\beta_2-\alpha_2\beta_5)\neq 0$}).
\label{eqn:5a}
\end{align}
Hence, if $\alpha_2/\alpha_4=1$ (resp. $\beta_2/\beta_5-\alpha_2/\alpha_4=1$),
 then the zero $r_2=0$
 (resp. the zeros $r_2=\pm\sqrt{-\alpha_4/\beta_5}$) of $\bar{\kappa}_3(r_2)$
 is (resp. are) double
 and the zeros $r_2=\pm\sqrt{-\alpha_4/\beta_5}$
 (resp. the zero $r_2=0$) of $\bar{\kappa}_{3b}(r_2)$
 are (resp. is) simple for any $b\in\Nset^3$
 since $\alpha_2/\alpha_4$ or $\beta_2/\beta_5\not\in\Qset$.
Thus, condition~(ii) of Proposition~\ref{prop:3a} holds
 if $\alpha_2/\alpha_4$ or $\beta_2/\beta_5-\alpha_2/\alpha_4=1$.

We assume that $2\alpha_2/\alpha_4\not\in\Zset_{\ge 2}$
 and $\beta_2/\beta_5-\alpha_2/\alpha_4\not\in\Zset_{\ge 1}$.
All conditions in \eqref{eqn:5a} hold,
 so that the zeros  $r_2=0$ and $\pm\sqrt{-\alpha_4/\beta_5}$ of $\bar{\kappa}_{3b}(r_2)$
 are simple, for any $b\in\Nset^3$.
Equation~\eqref{eqn:rhok} becomes
\[
\rho_3(r_2)
 =2\left(\frac{\beta_2}{\beta_5}-2\right)r_2^2+\frac{2(\alpha_2-\alpha_4)}{\beta_5}.
\]
Thus, $\deg(\kappa_{1\d})=3$, $\deg(\rho_3)=0$ or $2$,
 and $\kappa_{1\d}(r_2)$ and $\rho_3(r_2)$ have no common factor
 since otherwise $\rho_3(r_2)$ has roots at $r_2=0$ or $\pm\sqrt{-\alpha_4/\beta_5}$
 and consequently
 \[
\alpha_2/\alpha_4=1\quad\mbox{or}\quad
\beta_2/\beta_5-\alpha_2/\alpha_4=1.
\]
If $\beta_3\beta_5-\beta_2\beta_6\neq 0$,
 then $\deg(\kappa_{3\n})=4$ and Eq.~\eqref{eqn:prop3a} has the form \eqref{eqn:prop3b}
 with $\ell_a=3$, $\ell_b=0$ or $2$, $\ell_c=4$ and $b(\xi)/c(\xi)\not\in\Cset$.
On the other hand, if $\beta_3\beta_5-\beta_2\beta_6=0$ and Eq.~\eqref{eqn:thm2a} holds,
 then Eq.~\eqref{eqn:prop3a} has the form \eqref{eqn:prop3b}
 with $\ell_a=3$, $\ell_b,\ell_c=0$ or $2$ and $b(\xi)/c(\xi)\not\in\Cset$ since
\begin{align*}
&
(\alpha_2-\alpha_4)(\alpha_1\beta_5-\alpha_2\beta_6-\alpha_3\beta_2+\alpha_4\beta_3)
 -(\beta_2-2\beta_5)(\alpha_1\alpha_4-\alpha_2\alpha_3)\\
&=
(\alpha_1-\alpha_3)(\alpha_2\beta_5-\alpha_4\beta_2)
 -(\alpha_2-\alpha_4)(\alpha_2\beta_6-\alpha_4\beta_3)
 +(\alpha_1\alpha_4-\alpha_2\alpha_3)\beta_5\neq 0
\end{align*}
by \eqref{eqn:thm2a}.
Using Propositions~\ref{prop:3a}(iii) and \ref{prop:3b}(iii),
 we obtain Theorem~\ref{thm:main2}(i).

\begin{rmk}
\label{rmk:5a}
In the above case,
 we can also take $k=5$ in application of Theorem~$\ref{thm:c}$ and Proposition~$\ref{prop:3a}$
 but cannot improve our result, as in Remark~$\ref{rmk:4a}$.
Actually, we have
\begin{align*}
\kappa_{5\n}(r_2)=&\frac{120}{\beta_5^3}
((\beta_1\beta_5^2-(\beta_2\beta_4+\beta_3\beta_6)\beta_5+\beta_2\beta_6^2)r_2^6
-(\beta_5\beta_6\alpha_1+(\beta_4\beta_5-\beta_6^2)\alpha_2\\
&\qquad
 -(2\beta_2\beta_6-\beta_3\beta_5)\alpha_3
 -(2\beta_1\beta_5-\beta_2\beta_4-\beta_3\beta_6)\alpha_4))r_2^4\\
&
+(\alpha_4^2\beta_1+\alpha_3^2\beta_2-\alpha_3\alpha_4\beta_3-\alpha_2\alpha_4\beta_4
 -\alpha_1\alpha_3\beta_5\\
& \qquad
 +(2\alpha_2\alpha_3-\alpha_1\alpha_4)\beta_6)r_2^2
 +\alpha_3(\alpha_2\alpha_3-\alpha_1\alpha_4)),\\
\kappa_{5\d}(r_2)=&r_2^5\left(r_2^2+\frac{\alpha_4}{\beta_5}\right)^3,
\end{align*}
from which $n_1=3$, $\xi_{11}=0$, $\xi_{12}=\sqrt{-\alpha_4/\beta_5}$,
 $\xi_{13}=-\sqrt{-\alpha_4/\beta_5}$, $a_{11}=4$, $a_{12},a_{13}=2$ and $n_5=0$.
We compute \eqref{eqn:bkappa} as
\begin{align*}
\bar{\kappa}_{5b}(r_2)
  =&\biggl(\left(\frac{4\beta_2}{\beta_5}-(b_1+b_2+b_3)-5\right)r_2^2\\
&  -(b_2-b_3)\sqrt{-\frac{\alpha_4}{\beta_5}}r_2
 +\frac{4\alpha_2-\alpha_4(b_1+3)}{\beta_5}\biggr)
 r_2\left(r_2^2+\frac{\alpha_4}{\beta_5}\right),
\end{align*}
where $b=(b_1,b_2,b_3)\in\Nset^3$.
We see that the zero $r_2=0$
 $($resp. the zeros $r_2=\sqrt{-\alpha_4/\beta_5}$ and $-\sqrt{-\alpha_4/\beta_5})$
 of $\bar{\kappa}_{3b}(r_2)$
 is $($resp. are$)$ simple if and only if
\begin{align*}
\alpha_4(b_1+3)-4\alpha_2\neq 0\quad
(&\mbox{resp. $\alpha_4\beta_5(b_2+1)-2(\alpha_4\beta_2-\alpha_2\beta_5)$}\notag\\
& \mbox{and $\alpha_4\beta_5(b_3+1)-2(\alpha_4\beta_2-\alpha_2\beta_5)\neq 0$}).
\end{align*}
Hence, if $\alpha_2/\alpha_4=1$
 $($resp. $\beta_2/\beta_5-\alpha_2/\alpha_4=1)$, then the zero $r_2=0$
 $($resp. the zeros $r_2=\pm\sqrt{-\alpha_4/\beta_5})$of $\bar{\kappa}_3(r_2)$
 is $($resp. are$)$ double
 and the zeros $r_2=\pm\sqrt{-\alpha_4/\beta_5}$
 $($resp. the zero $r_2=0)$ of $\bar{\kappa}_{3b}(r_2)$
 are $($resp. is$)$ simple for any $b\in\Nset^3$. 
Thus, condition~{\rm(ii)} of Proposition~$\ref{prop:3a}$ holds
 if $\alpha_2/\alpha_4=1$ or $\beta_2/\beta_5-\alpha_2/\alpha_4=1$.
So we assume that $4\alpha_2/\alpha_4\not\in\Zset_{\ge 4}$
 and $2(\beta_2/\beta_5-\alpha_2/\alpha_4)\not\in\Zset_{\ge 2}$
 and repeat the above argument,
 but still need to exclude the case of $2\alpha_2/\alpha_4\in\Zset_{\ge 2}$
 and $\beta_2/\beta_5-\alpha_2/\alpha_4\not\in\Zset_{\ge 1}$.
We do not repeat such an argument for \eqref{eqn:PD4} below.
\end{rmk}

\subsubsection{Case of $\alpha_2,\alpha_1\alpha_4-\alpha_2\alpha_3,
 \alpha_1\beta_5+\alpha_2\beta_6-\alpha_3\beta_2-\alpha_4\beta_3\neq 0$
 and $\beta_2/\beta_5=\alpha_2/\alpha_4$}
We immediately have $\alpha_2/\alpha_4=\beta_2/\beta_5\not\in\Qset$ and
\[
\alpha_1\beta_5-\alpha_3\beta_2
 =\frac{\beta_5}{\alpha_4}(\alpha_1\alpha_4-\alpha_2\alpha_3)\neq 0.
\]
By \eqref{eqn:kdH} we have
\begin{align*}
&
\kappa_{1\n}(r_2)=\frac{\beta_2}{\beta_5}\neq 0,\quad
\kappa_{1\d}(r_2)=r_2,\\
&
\kappa_{3\n}(r_2)=\frac{6}{\beta_5^2}((\beta_3\beta_5-\beta_2\beta_6)r_2^2
 +\alpha_1\beta_5-\alpha_3\beta_2)\not\equiv 0,\quad
\kappa_{3\d}(r_2)=r_2^3\left(r_2^2+\frac{\alpha_4}{\beta_5}\right).
\end{align*}
Note that 
\[
\kappa_{3\n}\left(\pm\sqrt{-\frac{\alpha_4}{\beta_5}}\right)
 =\frac{6}{\beta_5}(\alpha_1\beta_5+\alpha_2\beta_6-\alpha_3\beta_2-\alpha_4\beta_3)\neq 0.
 \]
We have $n_1=1$, $\xi_{11}=0$, $a_{11}=2$,
 $n_3=2$, $\xi_{31}=\sqrt{-\alpha_4/\beta_5}$, $\xi_{32}=-\sqrt{-\alpha_4/\beta_5}$,
 and $a_{31}=a_{32}=1$,
 so that condition~(i) of Proposition~\ref{prop:3a} holds.
Thus, we obtain Theorem~\ref{thm:main2}(ii).

\subsubsection{Case of $\alpha_2,\alpha_1\alpha_4-\alpha_2\alpha_3,
 \beta_2/\beta_5-\alpha_2/\alpha_4\neq 0$
 and $\beta_6/\beta_5=\alpha_3/\alpha_4$}

By \eqref{eqn:kdH} we have
\[
\kappa_{1\n}(r_2)=\frac{\beta_2r_2^2+\alpha_2}{\beta_5}\not\equiv 0,\quad
\kappa_{1\d}(r_2)=r_2\left(r_2^2+\frac{\alpha_4}{\beta_5}\right)
\]
and
\[
\kappa_{3\n}(r_2)=\frac{6}{\beta_5^2}((\beta_3\beta_5-\beta_2\beta_6)r_2^2
 +\alpha_1\beta_5-\alpha_2\beta_6)\not\equiv 0,\quad
\kappa_{3\d}(r_2)=r_2^3\left(r_2^2+\frac{\alpha_4}{\beta_5}\right)
\]
if $\beta_6/\beta_5-\alpha_1/\alpha_2,
 \alpha_1\beta_5-\alpha_2\beta_6+\alpha_3\beta_2-\alpha_4\beta_3\neq 0$;
\[
\kappa_{3\n}(r_2)=\frac{6}{\beta_5^2}(\beta_3\beta_5-\beta_2\beta_6)\neq 0,\quad
\kappa_{3\d}(r_2)=r_2\left(r_2^2+\frac{\alpha_4}{\beta_5}\right)
\]
if $\beta_6/\beta_5-\alpha_1/\alpha_2=0$
 and $\alpha_1\beta_5-\alpha_2\beta_6+\alpha_3\beta_2-\alpha_4\beta_3\neq 0$;
 and
\[
\kappa_{3\n}(r_2)=\frac{6}{\beta_5^2}(\beta_3\beta_5-\beta_2\beta_6)\neq 0,\quad
\kappa_{3\d}(r_2)=r_2^3
\]
if $\alpha_1\beta_5-\alpha_2\beta_6+\alpha_3\beta_2-\alpha_4\beta_3=0$
 and $\beta_6/\beta_5-\alpha_1/\alpha_2\neq 0$.
Note that if $\beta_6/\beta_5\neq\alpha_1/\alpha_2$, then
\[
\alpha_1\beta_5-\alpha_2\beta_6+\alpha_3\beta_2-\alpha_4\beta_3
 \neq\alpha_3\beta_2-\alpha_4\beta_3
 =-\frac{\alpha_4}{\beta_5}(\beta_3\beta_5-\beta_2\beta_6)
\]
and that if $\beta_6/\beta_5-\alpha_1/\alpha_2=0$, then
\[
\alpha_1\beta_5-\alpha_2\beta_6+\alpha_3\beta_2-\alpha_4\beta_3
 =\alpha_3\beta_2-\alpha_4\beta_3
 =-\frac{\alpha_4}{\beta_5}(\beta_3\beta_5-\beta_2\beta_6),
\]
since $\beta_6/\beta_5=\alpha_3/\alpha_4$.
We separately discuss the three cases,
 when $\beta_6/\beta_5-\alpha_1/\alpha_2,\linebreak
  \alpha_1\beta_5-\alpha_2\beta_6+\alpha_3\beta_2-\alpha_4\beta_3\neq 0$;
 when $\beta_6/\beta_5-\alpha_1/\alpha_2=0$
 and $\alpha_1\beta_5-\alpha_2\beta_6+\alpha_3\beta_2-\alpha_4\beta_3\neq 0$;
 and when $\alpha_1\beta_5-\alpha_2\beta_6+\alpha_3\beta_2-\alpha_4\beta_3=0$
 and $\beta_6/\beta_5-\alpha_1/\alpha_2\neq 0$.

We discuss the first case. 
We have $n_1=0$, $\xi_{11}=0$, $a_{11}=2$ and $n_3=0$
 and compute \eqref{eqn:bkappa} as
\begin{align*}
\bar{\kappa}_{3b}(r_2)
=\biggl(\left(\frac{2\beta_2}{\beta_5}-b_1-1\right)r_2^2
  +\frac{2\alpha_2-\alpha_4(b_1+1)}{\beta_5}\biggr)r_2,
\end{align*}
where $b=b_1\in\Nset$.
The zero $r_2=0$ of $\bar{\kappa}_{3b}(r_2)$ is simple if and only if
\begin{align}
\alpha_4(b_1+1)-2\alpha_2\neq 0.
\label{eqn:5c1}
\end{align}
Hence, if $\alpha_2/\alpha_4=1$,
 then the zero $r_2=0$ of $\bar{\kappa}_3(r_2)$ is double,
 so that condition~(ii) of Proposition~\ref{prop:3a} holds.

Suppose that $2\alpha_2/\alpha_4\not\in\Zset_{\ge 2}$
 and Eq.~\eqref{eqn:thm2b} holds.
Then condition \eqref{eqn:5c1} holds,
 so that the zero  $r_2=0$ of $\bar{\kappa}_{3b}(r_2)$ is simple, for any $b\in\Nset$.
Equation~\eqref{eqn:rhok} becomes
\[
\rho_3(r_2)
 =2\left(\frac{\beta_2}{\beta_5}-1\right)r_2^2+\frac{2(\alpha_2-\alpha_4)}{\beta_5},
\]
so that $\deg(\kappa_{1\d})=3$, $\deg(\rho_3),\deg(\kappa_{3\n})=0$ or $2$,
 and $\kappa_{1\d}(r_2)$ and $\rho_3(r_2)$ have no common factor as in Section~5.1.1.
Thus, Eq.~\eqref{eqn:prop3a} has the form \eqref{eqn:prop3b}
 with $\ell_a=3$, $\ell_b,\ell_c=0$ or $2$ and $b(\xi)/c(\xi)\not\in\Cset$ since
\begin{align*}
&
(\alpha_2-\alpha_4)(\beta_3\beta_5-\beta_2\beta_6)
 -(\beta_2-\beta_5)(\alpha_1\beta_5-\alpha_2\beta_6)\\
&=
-\beta_5(\alpha_1(\beta_2-\beta_5) -\alpha_2(\beta_3-\beta_6)
 -\alpha_3\beta_6+\alpha_4\beta_3)\neq 0
\end{align*}
by \eqref{eqn:thm2b}.
Using Propositions~\ref{prop:3a}(iii) and \ref{prop:3b}(iii),
 we obtain Theorem~\ref{thm:main2}(iiia).

We turn to the second case and additionally assume that $\beta_2\neq 0$.
We have $n_1,n_3=0$ and
\[
\rho_3(r_2)=\frac{2(\beta_2r_2^3+\alpha_2)}{\beta_5},
\]
so that $\deg(\kappa_{1\d})=3$, $\deg(\rho_3)=2$, $\deg(\kappa_{3\n})=0$,
 and $\kappa_{1\d}(r_2)$ and $\rho_3(r_2)$ have no common factor as in Section~5.1.1.
Thus, Eq.~\eqref{eqn:prop3a} has the form \eqref{eqn:prop3b}
 with $\ell_a=3$, $\ell_b=2$, $\ell_c=0$ and $b(\xi)/c(\xi)\not\in\Cset$.
Using Propositions~\ref{prop:3a}(iii) and \ref{prop:3b}(iii),
 we obtain Theorem~\ref{thm:main2}(iiib). 

We turn to the third case.
We have $n_1=3$, $\xi_{11}=0$, $\xi_{12}=\sqrt{-\alpha_4/\beta_5}$, $\xi_{13}=-\sqrt{-\alpha_4/\beta_5}$,
 $a_{11}=2$, $a_{12},a_{13}=-1$ and $n_3=0$
 and compute \eqref{eqn:bkappa} as
\begin{align*}
\bar{\kappa}_{3b}(r_2)
=&\biggl(\left(\frac{2\beta_2}{\beta_5}-(b_1+b_2+b_3)+3\right)r_2^2\\
& -\sqrt{-\frac{\alpha_4}{\beta_5}}\,(b_2-b_3)r_2
 +\frac{2\alpha_2-\alpha_4(b_1+1)}{\beta_5}\biggr)r_2\left(r_2^2+\frac{\alpha_4}{\beta_5}\right),
\end{align*}
where $b=(b_1,b_2,b_3)\in\Nset^3$.
The zero $r_2=0$ (resp. the zeros $r_2=\pm\sqrt{\alpha_4/\beta_5}$)
 of $\bar{\kappa}_{3b}(r_2)$ is (resp. are) simple if and only if
\begin{align}
\alpha_4(b_1+1)-2\alpha_2\neq 0\quad
(& \mbox{resp. $\alpha_4\beta_5(b_2-2)-(\alpha_4\beta_2-\alpha_2\beta_5)$}\notag\\
& \mbox{and $\alpha_4\beta_5(b_3-2)-(\alpha_4\beta_2-\alpha_2\beta_5)\neq 0$}).
\label{eqn:5c2}
\end{align}
Hence, if $\alpha_2/\alpha_4=1$ (resp. $\beta_2/\beta_5-\alpha_2/\alpha_4=-1$),
 then the zero $r_2=0$ (resp. the zeros $r_2=\pm\sqrt{\alpha_4/\beta_5}$)
 of $\bar{\kappa}_3(r_2)$ is (resp. are) double
 and the zeros $r_2=\pm\sqrt{\alpha_4/\beta_5}$) of $\bar{\kappa}_3(r_2)$
 (resp the zero $r_2=0$) of $\bar{\kappa}_3(r_2)$ are (resp. is) simple for any $b\in\Nset^3$.
Thus, condition~(ii) of Proposition~\ref{prop:3a} holds
 if $\alpha_2/\alpha_4=1$ or $\beta_2/\beta_5-\alpha_2/\alpha_4=-1$.

Suppose that $2\alpha_2/\alpha_4\not\in\Zset_{\ge 2}$,
 $\beta_2/\beta_5-\alpha_2/\alpha_4\not\in\Zset_{\ge -1}$ and $\beta_2\neq 0$.
Then all conditions in \eqref{eqn:5c2} hold,
 so that the zeros  $r_2=0$ and $\pm\sqrt{-\alpha_4/\beta_5}$ of $\bar{\kappa}_{3b}(r_2)$
 are simple, for any $b\in\Nset^3$.
Equation~\eqref{eqn:rhok} becomes
\[
\rho_3(r_2)
 =\frac{2\beta_2}{\beta_5}r_2^2+\frac{2(\alpha_2-\alpha_4)}{\beta_5},
\]
so that $\deg(\kappa_{1\d})=3$, $\deg(\rho_3)=2$, $\deg(\kappa_{3\n})=0$,
 and $\kappa_{1\d}(r_2)$ and $\rho_3(r_2)$ have no common factor as in Section~5.1.1.
Thus, Eq.~\eqref{eqn:prop3a} has the form \eqref{eqn:prop3b}
 with $\ell_a=3$, $\ell_b=2$, $\ell_c=0$ and $b(\xi)/c(\xi)\not\in\Cset$.
Using Propositions~\ref{prop:3a}(iii) and \ref{prop:3b}(iii),
 we obtain Theorem~\ref{thm:main2}(iiic).

\subsubsection{Case of $\alpha_2,
 \beta_2/\beta_5-\alpha_2/\alpha_4,\beta_6/\beta_5-\alpha_3/\alpha_4\neq 0$
 and $\alpha_1\alpha_4=\alpha_2\alpha_3$}
By \eqref{eqn:kdH} we have
\[
\kappa_{1\n}(r_2)=\frac{\beta_2r_2^2+\alpha_2}{\beta_5}\not\equiv 0,\quad
\kappa_{1\d}(r_2)=r_2\left(r_2^2+\frac{\alpha_4}{\beta_5}\right)
\]
and
\begin{align*}
&
\kappa_{3\n}(r_2)=\frac{6\alpha_4}{\beta_5^2}((\beta_3\beta_5-\beta_2\beta_6)r_2^2
+\alpha_1\beta_5-\alpha_2\beta_6-\alpha_3\beta_2+\alpha_4\beta_3)\not\equiv 0,\\
&
\kappa_{3\d}(r_2)=r_2\left(r_2^2+\frac{\alpha_4}{\beta_5}\right)^2
\end{align*}
if $\alpha_1\beta_5-\alpha_2\beta_6-\alpha_3\beta_2+\alpha_4\beta_3\neq 0$; and
\[
\kappa_{3\n}(r_2)=\frac{6\alpha_4}{\beta_5^2}(\beta_3\beta_5-\beta_2\beta_6)r_2\neq 0,\quad
\kappa_{3\d}(r_2)=\left(r_2^2+\frac{\alpha_4}{\beta_5}\right)^2
\]
if $\alpha_1\beta_5-\alpha_2\beta_6-\alpha_3\beta_2+\alpha_4\beta_3=0$.
Note that if $\alpha_1\beta_5-\alpha_2\beta_6-\alpha_3\beta_2+\alpha_4\beta_3\neq 0$, then
\[
\kappa_{3\n}\left(\pm\sqrt{-\frac{\alpha_4}{\beta_5}}\right)
=\frac{6}{\beta_5^3}(\alpha_2\beta_5-\alpha_4\beta_2)(\alpha_4\beta_6-\alpha_3\beta_5)\neq 0
\]
and that if $\beta_3\beta_5-\beta_2\beta_6=0$, then by $\alpha_1\alpha_4-\alpha_2\alpha_3=0$
\[
\alpha_1\beta_5-\alpha_2\beta_6-\alpha_3\beta_2+\alpha_4\beta_3
=\frac{(\alpha_4\beta_2-\alpha_2\beta_5)(\alpha_4\beta_6-\alpha_3\beta_5)}{\alpha_4\beta_5}\neq 0.
\]
We separately discuss the two cases,
 when $\alpha_1\beta_5-\alpha_2\beta_6-\alpha_3\beta_2+\alpha_4\beta_3\neq 0$ or not.

We begin with the first case.
We have $n_1=2$, $\xi_{11}=\sqrt{-\alpha_4/\beta_5}$,
 $\xi_{12}=-\sqrt{-\alpha_4/\beta_5}$, $a_{11},a_{12}=1$ and $n_3=0$,
 and compute \eqref{eqn:bkappa} as
\begin{align*}
\bar{\kappa}_{3b}(r_2)
=\biggl(\left(\frac{2\beta_2}{\beta_5}-(b_1+b_2)\right)r_2^2
 -(b_1-b_2)\sqrt{-\frac{\alpha_4}{\beta_5}}r_2+\frac{2\alpha_2}{\beta_5}\biggr)
 \left(r_2^2+\frac{\alpha_4}{\beta_5}\right),
\end{align*}
where $b=(b_1,b_2)\in\Nset^2$.
The zeros $r_2=\sqrt{-\alpha_4/\beta_5}$ and $-\sqrt{-\alpha_4/\beta_5}$
 of $\bar{\kappa}_{3b}(r_2)$ are simple if and only if
\begin{align}
\alpha_4\beta_5b_1-(\alpha_4\beta_2-\alpha_2\beta_5)\neq 0
\quad\mbox{and}\quad
\alpha_4\beta_5b_2-(\alpha_4\beta_2-\alpha_2\beta_5)\neq 0,
\label{eqn:5b}
\end{align}
respectively.
Hence, if $\beta_2/\beta_5-\alpha_2/\alpha_4=1$,
 then the zeros $r_2=\pm\sqrt{-\alpha_4/\beta_5}$ of $\bar{\kappa}_3(r_2)$ are double,
 so that condition~(ii) of Proposition~\ref{prop:3a} holds.

Suppose that $\beta_2/\beta_5-\alpha_2/\alpha_4\not\in\Zset_{\ge 1}$
 and Eq.~\eqref{eqn:thm2c} holds.
Then both conditions in \eqref{eqn:5b} hold,
 so that the zeros  $r_2=\pm\sqrt{-\alpha_4/\beta_5}$ of $\bar{\kappa}_{3b}(r_2)$
 are simple, for any $b\in\Nset^2$.
Equation~\eqref{eqn:rhok} becomes
\[
\rho_3(r_2)
 =2\left(\frac{\beta_2}{\beta_5}-1\right)r_2^2+\frac{2\alpha_2}{\beta_5},
\]
so that $\deg(\kappa_{1\d})=3$, $\deg(\rho_3),\deg(\kappa_{3\n})=0$ or $2$,
 and $\kappa_{1\d}(r_2)$ and $\rho_3(r_2)$ have no common factor as in Section~5.1.1.
Thus, Eq.~\eqref{eqn:prop3a} has the form \eqref{eqn:prop3b}
 with $\ell_a=3$, $\ell_b,\ell_c=0$ or $2$ and $b(\xi)\c(\xi)\not\in\Cset$ since
\begin{align*}
&
(\beta_2-\beta_5)(\alpha_1\beta_5-\alpha_2\beta_6-\alpha_3\beta_2+\alpha_4\beta_3)-\alpha_2(\beta_3\beta_5-\beta_2\beta_6)\\
&=(\beta_2-\beta_5)(\alpha_1\beta_5-\alpha_3\beta_2+\alpha_4\beta_3)
 -(\beta_3-\beta_6)\alpha_2\beta_5\neq 0
\end{align*}
by \eqref{eqn:thm2c}.
Using Propositions~\ref{prop:3a}(iii) and \ref{prop:3b}(iii),
 we obtain Theorem~\ref{thm:main2}(iva).

We turn to the second case.
We have $n_1=3$, $\xi_{11}=0$, $\xi_{12}=\sqrt{-\alpha_4/\beta_5}$,
 $\xi_{13}=-\sqrt{-\alpha_4/\beta_5}$, $a_{11}=-1$, $a_{12},a_{13}=1$ and $n_3=0$,
 and compute \eqref{eqn:bkappa} as
\begin{align*}
\bar{\kappa}_{3b}(r_2)
=&\biggl(\left(\frac{2\beta_2}{\beta_5}-(b_1+b_2+b_3)+2\right)r_2^2\\
& -(b_2-b_3)\sqrt{-\frac{\alpha_4}{\beta_5}}r_2+\frac{2\alpha_2-\alpha_4(b_1-2)}{\beta_5}\biggr)
 r_2\left(r_2^2+\frac{\alpha_4}{\beta_5}\right),
\end{align*}
where $b=(b_1,b_2,b_3)\in\Nset^3$.
The zero $r_2=0$ (resp. the zeros $r_2=\sqrt{-\alpha_4/\beta_5}$ and $-\sqrt{-\alpha_4/\beta_5}$)
 of $\bar{\kappa}_{3b}(r_2)$
 is (resp. are) simple if and only if
\begin{align}
\alpha_4(b_1-2)-2\alpha_2\neq 0\quad
(&\mbox{resp. $\alpha_4\beta_5 b_2-(\alpha_4\beta_2-\alpha_2\beta_5)$}\notag\\
& \mbox{and $\alpha_4\beta_5 b_3-(\alpha_4\beta_2-\alpha_2\beta_5)\neq 0$}).
\label{eqn:5c}
\end{align}
Hence, if $\alpha_2/\alpha_4=-\tfrac{1}{2}$ (resp. $\beta_2/\beta_5-\alpha_2/\alpha_4=1$),
 then the zero $r_2=0$
 (resp. the zeros $r_2=\pm\sqrt{-\alpha_4/\beta_5}$) of $\bar{\kappa}_3(r_2)$
 is (resp. are) double
 and the zeros $r_2=\pm\sqrt{-\alpha_4/\beta_5}$
 (resp. the zero $r_2=0$) of $\bar{\kappa}_{3b}(r_2)$
 are (resp. is) simple for any $b\in\Nset^3$. 
Thus, condition~(ii) of Proposition~\ref{prop:3a} holds
 if $\alpha_2/\alpha_4=-\tfrac{1}{2}$ or $\beta_2/\beta_5-\alpha_2/\alpha_4=1$.

Suppose that $2\alpha_2/\alpha_4\not\in\Zset_{\ge -1}$
 and $\beta_2/\beta_5-\alpha_2/\alpha_4\not\in\Zset_{\ge 1}$.
Then all conditions in \eqref{eqn:5c} hold,
 so that the zeros  $r_2=0$ and $\pm\sqrt{-\alpha_4/\beta_5}$ of $\bar{\kappa}_{3b}(r_2)$
 are simple, for any $b\in\Nset^3$.
Equation~\eqref{eqn:rhok} becomes
\[
\rho_3(r_2)
 =\left(\frac{2\beta_2}{\beta_5}-1\right)r_2^2+\frac{2\alpha_2+\alpha_4}{\beta_5},
\]
so that $\deg(\kappa_{1\d})=3$, $\deg(\rho_3)=0$ or $2$, $\deg(\kappa_{3\n})=1$,
 and $\kappa_{1\d}(r_2)$ and $\rho_3(r_2)$ have no common factor
 as in Section~5.1.1.
Thus, Eq.~\eqref{eqn:prop3a} has the form \eqref{eqn:prop3b}
 with $\ell_a=3$, $\ell_b=0$ or $2$, $\ell_c=1$ and $b(\xi)/c(\xi)\not\in\Cset$.
Using Propositions~\ref{prop:3a}(iii) and \ref{prop:3b}(iii),
 we obtain Theorem~\ref{thm:main2}(ivb).
 
\subsubsection{Case of $\alpha_1\beta_2,\beta_6/\beta_5-\alpha_3/\alpha_4\neq 0$ and $\alpha_2=0$}
We immediately have $\beta_2/\beta_5\not\in\Qset$.
By \eqref{eqn:kdH} we have
\begin{align*}
&
\kappa_{1\n}(r_2)=\frac{\beta_2r_2}{\beta_5}\not\equiv 0,\quad
\kappa_{1\d}(r_2)=r_2^2+\frac{\alpha_4}{\beta_5},\\
&
\kappa_{3\n}(r_2)
 =\frac{6}{\beta_5^2}((\beta_3\beta_5-\beta_2\beta_6)r_2^4
 +(\alpha_1\beta_5+\alpha_4\beta_3-\alpha_3\beta_2)r_2^2+\alpha_1\alpha_4)\not\equiv 0,\\
&
\kappa_{3\d}(r_2)=r_2^3\left(r_2^2+\frac{\alpha_4}{\beta_5}\right)^2,
\end{align*}
from which $n_1=2$, $\xi_{11}=\sqrt{-\alpha_4/\beta_5}$, $\xi_{12}=-\sqrt{-\alpha_4/\beta_5}$, 
 $a_{11},a_{12}=1$, $n_3=1$, $\xi_{31}=0$ and $a_{31}=3$.
Note that $\kappa_{3\n}(r_2)$ has no zero at $r_2=\pm\sqrt{-\alpha_4/\beta_5}$ since
\[
\kappa_{3\n}\left(\pm\sqrt{-\frac{\alpha_4}{\beta_5}}\right)
=\frac{6\alpha_4\beta_2}{\beta_5^4}(\alpha_3\beta_5-\alpha_4\beta_6)\neq 0.
\]
We compute \eqref{eqn:bkappa} as
\[
\bar{\kappa}_{3b}(r_2)
=\biggl(\left(\frac{2\beta_2}{\beta_5}-(b_1+b_2)\right)r_2
 -(b_1-b_2)\sqrt{-\frac{\alpha_4}{\beta_5}}\biggr)
 \left(r_2^2+\frac{\alpha_4}{\beta_5}\right),
\]
where $b=(b_1,b_2)\in\Nset^2$.
We see by $\beta_2/\beta_5\not\in\Qset$
 that the zeros $r_2=\pm\sqrt{-\alpha_4/\beta_5}$ of $\bar{\kappa}_{3b}(r_2)$
 are simple for any $b\in\Nset^2$.
Equation~\eqref{eqn:rhok} becomes
\[
\rho_3(r_2)
 =2\left(\frac{\beta_2}{\beta_5}-2\right)r_2^2-\frac{2\alpha_4}{\beta_5},
\]
so that $\deg(\kappa_{1\d})=2$, $\deg(\rho_3)=2$
 and $\kappa_{1\d}(r_2)$ and $\rho_3(r_2)$  have no common factor since
\[
\rho_3\left(\pm\sqrt{-\frac{\alpha_4}{\beta_5}}\right)
=\frac{2\alpha_4(\beta_5-\beta_2)}{\beta_5^2}\neq 0.
\]
If $\beta_3\beta_5-\beta_2\beta_6\neq 0$,
 then $\deg(\kappa_{3\n})=4$
 and Eq.~\eqref{eqn:prop3a} has the form \eqref{eqn:prop3b}
 with $\ell_a=3$, $\ell_b=2$, $\ell_c\le 4$ and $b(\xi)/c(\xi)\not\in\Cset$.
On the other hand,
 if $\beta_3\beta_5-\beta_2\beta_6=0$ and $(\alpha_1-\alpha_3)\beta_2+\alpha_4\beta_3\neq 0$,
 then $\deg(\kappa_{3\n})=0$ or $2$,
 and Eq.~\eqref{eqn:prop3a} has the form \eqref{eqn:prop3b}
 with $\ell_a=3$, $\ell_b=2$, $\ell_c=0$ or $2$ and $b(\xi)/c(\xi)\not\in\Cset$ since
\[
\alpha_4(\alpha_1\beta_5+\alpha_4\beta_3-\alpha_3\beta_2)+\alpha_1\alpha_4(\beta_2-\beta_5)
=\alpha_4((\alpha_1-\alpha_3)\beta_2+\alpha_4\beta_3)\neq 0.
\]
Using Propositions~\ref{prop:3a}(iii) and \ref{prop:3b}(iii),
 we obtain Theorem~\ref{thm:main2}(v).

\subsubsection{Case of $\beta_2,\beta_3/\beta_2-\alpha_3/\alpha_4\neq 0$
 and $\alpha_1,\alpha_2=0$}
We immediately have $\beta_2/\beta_5\not\in\Qset$.
By \eqref{eqn:kdH} we have
\[
\kappa_{1\n}(r_2)=\frac{\beta_2r_2}{\beta_5}\not\equiv 0,\quad
\kappa_{1\d}(r_2)=r_2^2+\frac{\alpha_4}{\beta_5}
\]
and
\begin{align*}
&
\kappa_{3\n}(r_2)
 =\frac{6}{\beta_5^2}((\beta_3\beta_5-\beta_2\beta_6)r_2^2
 +\alpha_4\beta_3-\alpha_3\beta_2)\not\equiv 0,\quad
\kappa_{3\d}(r_2)=r_2\left(r_2^2+\frac{\alpha_4}{\beta_5}\right)^2
\end{align*}
if $\beta_6/\beta_5-\alpha_3/\alpha_4\neq 0$; and
\begin{align*}
\kappa_{3\n}(r_2)=\frac{6\alpha_1}{\beta_5}(\beta_3\beta_5-\beta_2\beta_6)\neq 0,\quad
\kappa_{3\d}(r_2)=r_2\left(r_2^2+\frac{\alpha_4}{\beta_5}\right)
\end{align*}
if $\beta_6/\beta_5-\alpha_3/\alpha_4=0$.
Note that if $\beta_6/\beta_5-\alpha_3/\alpha_4=0$, then 
\[
\beta_3\beta_5-\beta_2\beta_6
=\beta_2\beta_5\left(\frac{\beta_3}{\beta_2}-\frac{\alpha_3}{\alpha_4}\right)\neq 0.
\]
Whether $\beta_6/\beta_5-\alpha_3/\alpha_4\neq 0$ or not,
 we have $n_3=1$, $\xi_{31}=0$ and $a_{31}=1$,
 so that condition~(i) of Proposition~\ref{prop:3a} holds.
This yields Theorem~\ref{thm:main2}(vi).

\subsection{Case of $\alpha_4=0$, $\beta_5\neq 0$,
 and $\alpha_2\neq 0$ or $\beta_2/\beta_5\not\in\Qset$}

We next consider the case of $\alpha_4=0$, $\beta_5\neq 0$,
 and $\alpha_2\neq 0$ or $\beta_2/\beta_5\not\in\Qset$.
From \eqref{eqn:kdH} we see that  $\deg(\kappa_{1\d})>\deg(\kappa_{1\n})$, and compute
\[
\Omega(r_2)=\exp\left(\int\frac{\beta_2r_2^2+\alpha_2}{\beta_5r_2^3}\d r_2\right)
=\exp\left(-\frac{\alpha_2}{2\beta_5r_2^2}\right)r_2^{\beta_2/\beta_5},
\]
so that condition~(H1) holds
 since $\alpha_2\neq 0$ or $\beta_2/\beta_5\not\in\Qset$.
We separately consider cases~(b1) and (b2) stated at the beginning of this section
 and check the hypotheses of Proposition~\ref{prop:3a}.

\subsubsection{Case of $\alpha_2\alpha_3\neq 0$}

By \eqref{eqn:kdH} we have
\begin{align*}
&
\kappa_{1\n}(r_2)=\frac{\beta_2r_2^2+\alpha_2}{\beta_5}\not\equiv 0,\quad
\kappa_{1\d}(r_2)=r_2^3,\\
&
\kappa_{3\n}(r_2)
 =\frac{6}{\beta_5^2}((\beta_3\beta_5-\beta_2\beta_6)r_2^4
 +(\alpha_1\beta_5-\alpha_2\beta_6-\alpha_3\beta_2)r_2^2-\alpha_2\alpha_3)\not\equiv 0,\\
&
\kappa_{3\d}(r_2)=r_2^7,
\end{align*}
from which $n_1=1$, $\xi_{11}=0$, $a_{11}=4$ and $n_3=0$.
We compute \eqref{eqn:bkappa} as
\[
\bar{\kappa}_{3b}(r_2)
=\left(\left(\frac{2\beta_2}{\beta_5}-b_1-3\right)r_2^2+\frac{2\alpha_2}{\beta_5}\right)r_2,
\]
where $b=b_1\in\Nset$.
The zero $r_2=0$ of $\bar{\kappa}_{3b}(r_2)$ is simple for any $b\in\Nset$.
Equation~\eqref{eqn:rhok} becomes
\[
\rho_3(r_2)
 =2\left(\frac{\beta_2}{\beta_5}-2\right)r_2^2+\frac{2\alpha_2}{\beta_5},
\]
so that $\deg(\kappa_{1\d})=3$, $\deg(\rho_3)=0$ or $2$,
 and $\kappa_{1\d}(r_2)$ and $\rho_3(r_2)$  have no common factor.
If $\beta_3\beta_5-\beta_2\beta_6\neq 0$, then $\deg(\kappa_{3\n})=4$
 and Eq.~\eqref{eqn:prop3a} has the form \eqref{eqn:prop3b}
 with $\ell_a=3$, $\ell_b\le 2$, $\ell_c=4$ and $b(\xi)/c(\xi)\not\in\Cset$.
On the other hand, if $\beta_3\beta_5-\beta_2\beta_6=0$
 and $(\alpha_1-2\alpha_3)\beta_5-\alpha_2\beta_6\neq 0$, 
 then Eq.~\eqref{eqn:prop3a} has the form \eqref{eqn:prop3b}
 with $\ell_a=3$, $\ell_b,\ell_c\le 2$ and $b(\xi)/c(\xi)\not\in\Cset$ since
\[
\alpha_2(\alpha_1\beta_5-\alpha_2\beta_6-\alpha_3\beta_2)+\alpha_2\alpha_3(\beta_2-2\beta_5)
=\alpha_2((\alpha_1-2\alpha_3)\beta_5-\alpha_2\beta_6)\neq 0.
\]
Using Propositions~\ref{prop:3a}(iii) and ~\ref{prop:3b}(iii),
 we obtain Theorem~\ref{thm:main2}(vii).

\subsubsection{Case of $\alpha_1\alpha_2\neq 0$ and $\alpha_3=0$}
By \eqref{eqn:kdH} we have
\[
\kappa_{1\n}(r_2)=\frac{\beta_2r_2^2+\alpha_2}{\beta_5}\not\equiv 0,\quad
\kappa_{1\d}(r_2)=r_2^3
\]
and
\[
\kappa_{3\n}(r_2)=\frac{6}{\beta_5^2}
 ((\beta_3\beta_5-\beta_2\beta_6)r_2^2+\alpha_1\beta_5-\alpha_2\beta_6)\not\equiv 0,\quad
\kappa_{3\d}(r_2)=r_2^5
\]
if $\beta_6/\beta_5-\alpha_1/\alpha_2\neq 0$; and
\[
\kappa_{3\n}(r_2)=\frac{6}{\beta_5^2}(\beta_3\beta_5-\beta_2\beta_6)\neq 0,\quad
\kappa_{3\d}(r_2)=r_2^3
\]
if $\beta_6/\beta_5-\alpha_1/\alpha_2= 0$ and $\beta_3\beta_5-\beta_2\beta_6\neq 0$.
We separately discuss the two cases,
 when $\beta_6/\beta_5-\alpha_1/\alpha_2\neq 0$;
 and when $\beta_6/\beta_5-\alpha_1/\alpha_2= 0$ and $\beta_3\beta_5-\beta_2\beta_6\neq 0$.

We begin with the first case.
We have $n_1=1$, $\xi_{11}=0$, $a_{11}=2$ and $n_3=0$, and compute \eqref{eqn:bkappa} as
\[
\bar{\kappa}_{3b}(r_2)
=\left(\left(\frac{2\beta_2}{\beta_5}-b_1-1\right)r_2^2+\frac{2\alpha_2}{\beta_5}\right)r_2,
\]
where $b=b_1\in\Nset$.
The zero $r_2=0$ of $\bar{\kappa}_{3b}(r_2)$ is simple for any $b\in\Nset$.
Equation~\eqref{eqn:rhok} becomes
\[
\rho_3(r_2)
 =2\left(\frac{\beta_2}{\beta_5}-1\right)r_2^2+\frac{2\alpha_2}{\beta_5},
\]
so that $\deg(\kappa_{1\d})=3$, $\deg(\rho_3),\deg(\kappa_{3\n})=0$ or $2$
 and $\kappa_{1\d}(r_2)$ and $\rho_3(r_2)$ have no common factor.
If $\alpha_1(\beta_2-\beta_5)-\alpha_2(\beta_3-\beta_6)\neq 0$,
 then Eq.~\eqref{eqn:prop3a} has the form \eqref{eqn:prop3b}
 with $\ell_a=3$, $\ell_b,\ell_c=0$ or $2$ and $b(\xi)/c(\xi)\not\in\Cset$ since
\begin{align*}
&
(\beta_2-\beta_5)(\alpha_1\beta_5-\alpha_2\beta_6)-\alpha_2(\beta_3\beta_5-\beta_2\beta_6)\\
&
=\beta_5(\alpha_1(\beta_2-\beta_5)-\alpha_2(\beta_3-\beta_6))\neq 0.
\end{align*}
Using Propositions~\ref{prop:3a}(iii) and \ref{prop:3b}(iii),
 we obtain Theorem~\ref{thm:main2}(viiia).
 
We turn to the second case and additionally assume that $\beta_2\neq 0$.
We have $n_1,n_3=0$ and
\[
\rho_3(r_2)=\frac{2(\beta_2r_2^3+\alpha_2)}{\beta_5},
\]
so that $\deg(\kappa_{1\d})=3$, $\deg(\rho_3)=2$, $\deg(\kappa_{3\n})=0$,
 and $\kappa_{1\d}(r_2)$ and $\rho_3(r_2)$ have no common factor.
Thus, Eq.~\eqref{eqn:prop3a} has the form \eqref{eqn:prop3b}
 with $\ell_a=3$, $\ell_b=2$, $\ell_c=0$ and $b(\xi)/c(\xi)\not\in\Cset$.
Using Propositions~\ref{prop:3a}(iii) and \ref{prop:3b}(iii),
 we obtain Theorem~\ref{thm:main2}\linebreak(viiib)
 and complete the proof of Theorem~\ref{thm:main2}.
\qed


\section*{Acknowledgements}
This work was partially supported by the JSPS KAKENHI Grant Number JP22H01138.


\appendix
\renewcommand{\theequation}{\Alph{section}.\arabic{equation}}

\section{Proof of Proposition~\ref{prop:3b}}

In this appendix we give a proof of Proposition~\ref{prop:3b}.
The following proposition, on which an algorithm of Rothstein \cite{B05,R76,R77}
 to determine whether given elementary functions have elementary antiderivative are based,
 plays a key role in the proof.

\begin{prop}
\label{prop:a1}
Suppose that $a(\xi),b(\xi)$ have no common factor
 and let $r(\xi),s(\xi)$ be polynomials of $\xi$
 such that $c(\xi)=a(\xi)r(\xi)+b(\xi)s(\xi)$ and either $s(\xi)\equiv 0$ or $\deg(s)<\deg(a)$.
If Eq.~\eqref{eqn:prop3b} has a polynomial solution $z(\xi)$,
 then $\zeta(\xi)=(z(\xi)-s(\xi))/a(\xi)$ is a polynomial solution to
\begin{equation}
a(\xi)\zeta'+(a'(\xi)+b(\xi))\zeta=r(\xi)-s'(\xi).
\label{eqn:propa1}
\end{equation}
\end{prop}

See, e.g., Section~6.4 of \cite{B05} for a proof of Proposition~\ref{prop:a1}.
Note that if $a(\xi),b(\xi)$ have no common factor,
 then there exist such polynomials $r(\xi),s(\xi)$ as in the statement of  Proposition~\ref{prop:a1}.
We also write $z(\xi)=a(\xi)\zeta(\xi)+s(\xi)$
 with $\zeta(\xi),s(\xi)\in\Cset[\xi]$ to obtain
\[
a(\xi)(a'(\xi)\zeta(\xi)+a(\xi)\zeta'(\xi)+s'(\xi)+b(\xi)\zeta(\xi))+b(\xi)s(\xi)=c(\xi)
\]
from \eqref{eqn:prop3b}.

\begin{proof}[Proof of Proposition~$\ref{prop:3b}$]
Suppose that the hypotheses of the proposition hold.
We easily see that Eq.~\eqref{eqn:prop3b} has no constant solution.
Let $c_0$ denote the leading coefficient of $c(\xi)$ and let
\begin{equation}
z(\xi)=z_0\xi^\ell+\cdots 
\label{eqn:a1z}
\end{equation}
denote the polynomial solution to \eqref{eqn:prop3b}
 for some $z_0\neq 0$ and $\ell\in\Nset$.
Substituting \eqref{eqn:a1z} into \eqref{eqn:prop3b}, we have
\[
\ell z_0\xi^{\ell_a+\ell-1}+\cdots+b_0z_0\xi^{\ell_b+\ell}+\cdots=c_0\xi^{\ell_c}+\cdots.
\]
Hence, $\ell_c=\max(\ell_a+\ell-1,\ell_b+\ell)$, or $\ell_a+\ell-1=\ell_b+\ell>\ell_c$ and
\[
b_0+\ell=0
\]
for some $\ell\in\Nset$.
Thus, we obtain part~(i).
If $a(\xi),b(\xi)$ have a common factor $d(\xi)$,
 then Eq.~\eqref{eqn:prop3b} is rewritten as
\[
\frac{a(\xi)}{d(\xi)}z'(\xi)+\frac{b(\xi)}{d(\xi)}z(\xi)=\frac{c(\xi)}{d(\xi)},
\]
of which the left hand side is a polynomial.
This means part~(ii).

We turn to the proof of part~(iii).
If $\ell_a>\ell_b+1$,
 then $\deg(z)=\ell_c-\ell_a+1$ and by Proposition~\ref{prop:a1}
 Eq.~\eqref{eqn:propa1} has an $(\ell_c-2\ell_a+1)$th-order polynomial solution,
 so that $\ell_c\ge 2\ell_a-1>\ell_a+\ell_b$.
If $\ell_a<\ell_b+1$,
 then $\deg(z)=\ell_c-\ell_b$ and by Proposition~\ref{prop:a1}
 Eq.~\eqref{eqn:propa1} has an $(\ell_c-\ell_a-\ell_b)$th-order polynomial solution,
 so that $\ell_c\ge\ell_a+\ell_b>2\ell_a-1$.
We complete the proof.
\end{proof}


\end{document}